\documentclass[10pt, a4paper]{amsart}
\usepackage{amscd,amsmath,amssymb,amsfonts,amsthm,ascmac}
\usepackage{enumerate,mathrsfs,stmaryrd,latexsym,mathtools} 
\usepackage[all]{xy}
\usepackage[top=27truemm,bottom=21truemm,left=26truemm,right=26truemm]{geometry}
\usepackage{graphicx}

\newtheorem{thm}{Theorem}[section]
\newtheorem{defi}[thm]{Definition}
\newtheorem{rem}[thm]{Remark}
\newtheorem{prop}[thm]{Proposition}
\newtheorem{ex}[thm]{Example}
\newtheorem{cor}[thm]{Corollary}
\newtheorem{lem}[thm]{Lemma}

\newtheorem*{clm}{Claim}
\newtheorem*{ack}{Acknowledgements}
\newtheorem*{theor}{Theorem}
\newtheorem*{coro}{Corollary}

\newcommand{\wt}{\operatorname{wt}}
\newcommand{\Ker}{\operatorname{Ker}}
\newcommand{\Ind}{\operatorname{Ind}}
\newcommand{\Res}{\operatorname{Res}}
 \makeatletter
    
    \@addtoreset{equation}{section}
  \makeatother
\title[A COMPARISON OF NEWTON-OKOUNKOV POLYTOPES OF SCHUBERT VARIETIES]{\fontsize{11pt}{11pt}\selectfont A COMPARISON OF NEWTON-OKOUNKOV POLYTOPES OF SCHUBERT VARIETIES}
\date{}
\author[N. Fujita]{\fontsize{10pt}{10pt}\selectfont Naoki Fujita}
\author[H. Oya]{\fontsize{10pt}{10pt}\selectfont Hironori Oya}
\begin{document}
\address[Naoki Fujita]{Department of Mathematics, Tokyo Institute of Technology, 2-12-1 Oh-okayama, Meguro-ku, Tokyo 152-8551, Japan}
\email{fujita.n.ac@m.titech.ac.jp}
\address[Hironori Oya]{Graduate School of Mathematical Sciences, The University of Tokyo, Komaba, Tokyo, 153-8914, Japan}
\email{oya@ms.u-tokyo.ac.jp}

\thanks{The work of the first author was supported by Grant-in-Aid for JSPS Fellows (No.\ 16J00420). The work of the second author was supported by Grant-in-Aid for JSPS Fellows (No.\ 15J09231) and the Program for Leading Graduate Schools, MEXT, Japan.}
\begin{abstract}
A Newton-Okounkov body is a convex body constructed from a polarized variety with a valuation on its function field. Kaveh (resp., the first author and Naito) proved that the Newton-Okounkov body of a Schubert variety associated with a specific valuation is identical to the Littelmann string polytope (resp., the Nakashima-Zelevinsky polyhedral realization) of a Demazure crystal. These specific valuations are defined algebraically to be the highest term valuations with respect to certain local coordinate systems on a Bott-Samelson variety. Another class of valuations, which is geometrically natural, arises from some sequence of subvarieties of a polarized variety. In this paper, we show that the highest term valuation used by Kaveh (resp., by the first author and Naito) and the valuation coming from a sequence of specific subvarieties of the Schubert variety are identical on a perfect basis with some positivity properties. The existence of such a perfect basis follows from a categorification of the negative part of the quantized enveloping algebra. As a corollary, we prove that the associated Newton-Okounkov bodies coincide through an explicit affine transformation.
\end{abstract}
\maketitle
\setcounter{tocdepth}{2}
\tableofcontents
\section{Introduction}

A Newton-Okounkov body $\Delta(X, \mathcal{L}, v)$ is a convex body constructed from a polarized variety $(X, \mathcal{L})$ with a valuation $v$ on its function field $\mathbb{C}(X)$; this generalizes the notion of Newton polytope for a toric variety. The theory of Newton-Okounkov bodies was introduced by Okounkov  {\cite{Oko1, Oko2, Oko3}} and afterward developed independently by Kaveh-Khovanskii \cite{KK1} and by Lazarsfeld-Mustata \cite{LM}. A remarkable fact is that the theory of Newton-Okounkov bodies of Schubert varieties is deeply connected with representation theory \cite{FaFL, FeFL, FN, Kav, Kir}. For instance, Kaveh \cite{Kav} (resp., the first author and Naito \cite{FN}) showed that the Littelmann string polytope constructed from the Littelmann string parametrization for a Demazure crystal (resp., the Nakashima-Zelevinsky polyhedral realization constructed from the Kashiwara embedding of a Demazure crystal) is identical to the Newton-Okounkov body of a Schubert variety with respect to a specific valuation, which is defined algebraically to be the highest term valuation with respect to a certain local coordinate system on a Bott-Samelson variety (cf.~\cite{Fuj}). There are valuations which arise naturally from geometric data of $X$, more precisely, some sequences of subvarieties of $X$. The class of such valuations includes many interesting examples, and many people have been focused on Newton-Okounkov bodies with respect to such valuations (see, for instance, \cite{KLM} and \cite{LM}). In this paper, we show that the valuation used by Kaveh (resp., by the first author and Naito) and the one coming from a sequence of specific subvarieties of the Schubert variety are identical on a perfect basis with some positivity properties. 

To be more precise, let $X$ be an irreducible normal projective variety over $\mathbb{C}$ of complex dimension $r$, and $\mathcal{L}$ a very ample line bundle on $X$. We consider a sequence of irreducible closed subvarieties \[X_\bullet \colon X_r \subset X_{r-1} \subset \cdots \subset X_0 = X\] such that $\dim_\mathbb{C}(X_k) = r-k$ for $0 \le k \le r$, and assume that $X_k$ is a normal subvariety of $X_{k-1}$ for $1 \le k \le r$. By the normality assumption, there exists a collection $u_1, \ldots, u_r$ of rational functions on $X$ such that the restriction $u_k |_{X_{k-1}}$ is a not identically zero rational function on $X_{k-1}$ that has a zero of first order on the hypersurface $X_k$ for every $k$. Out of such a collection $u_1, \ldots, u_r$ of rational functions, we construct a valuation $v_{X_\bullet} \colon \mathbb{C}(X) \setminus \{0\} \rightarrow \mathbb{Z}^r$, $f \mapsto (a_1, \ldots, a_r)$, as follows. The first coordinate $a_1$ is the order of vanishing of $f$ on $X_1$. Then we have $(u_1 ^{-a_1} f)|_{X_1} \in \mathbb{C}(X_1) \setminus \{0\}$, and the second coordinate $a_2$ is the order of vanishing of $(u_1 ^{-a_1} f)|_{X_1}$ on $X_2$. Continuing in this way, we define all $a_k$. This is the definition of $v_{X_\bullet}$. The Newton-Okounkov body $\Delta(X, \mathcal{L}, v_{X_\bullet})$ inherits information about algebraic, geometric, and combinatorial properties of $X$; for instance, the Newton-Okounkov body $\Delta(X, \mathcal{L}, v_{X_\bullet})$ encodes numerical equivalence information of the line bundle $\mathcal{L}$ (see \cite{LM}). In addition, by \cite[Theorem 1]{And}, we can systematically construct a series of toric degenerations of $X$ under the assumption that the associated semigroup $\Gamma_{X_\bullet}(H^0(X, \mathcal{L}))$ (see \cite[Sections 2 and 3]{And} for the definition) is finitely generated. In the case that $X$ is a Schubert variety and $X_\bullet$ is a sequence of specific subvarieties of the Schubert variety, this semigroup $\Gamma_{X_\bullet}(H^0(X, \mathcal{L}))$ is identical to the semigroup $S(X, \mathcal{L}, v_{X_\bullet}, \tau)$ that we will define in Definition \ref{Newton-Okounkov convex body}. It is natural to ask whether the valuation used by Kaveh (resp., by the first author and Naito) can be realized as a valuation of the form $v_{X_\bullet}$. This question was suggested by Kaveh in \cite[Introduction (after Theorem 1)]{Kav}. Our main result in this paper gives an answer to this question.

Let $G$ be a connected, simply-connected semisimple algebraic group over $\mathbb{C}$, $\mathfrak{g}$ its Lie algebra, $W$ the Weyl group, and $s_i \in W$, $i \in I$, the simple reflections, where $I$ denotes an index set for the vertices of the Dynkin diagram. Choose a Borel subgroup $B \subset G$, and denote by $X(w) \subset G/B$ the Schubert variety corresponding to $w \in W$. A dominant integral weight $\lambda$ gives a line bundle $\mathcal{L}_\lambda$ on $G/B$; by restricting this bundle, we obtain a line bundle on $X(w)$, which we denote by the same symbol $\mathcal{L}_\lambda$. From the Borel-Weil type theorem, we know that the space $H^0(X(w), \mathcal{L}_\lambda)$ of global sections is a $B$-module isomorphic to the dual module $V_w (\lambda)^\ast$ of the Demazure module $V_w (\lambda)$ corresponding to $\lambda$ and $w$. Let ${\bf i} = (i_1, \ldots, i_r) \in I^r$ be a reduced word for $w$, and set $w_{\ge k} \coloneqq s_{i_k} s_{i_{k+1}} \cdots s_{i_r}$, $w_{\le k} \coloneqq s_{i_1} s_{i_2} \cdots s_{i_k}$ for $1 \le k \le r$. Then we obtain two sequences of subvarieties of $X(w)$ which satisfy the conditions above:
\begin{align*}
&X(w_{\ge \bullet}) \colon X(e) \subset X(w_{\ge r}) \subset X(w_{\ge r-1}) \subset \cdots \subset X(w_{\ge 2}) \subset X(w_{\ge 1}) = X(w)\ {\rm and}\\
&X(w_{\le \bullet}) \colon X(e) \subset X(w_{\le 1}) \subset X(w_{\le 2}) \subset \cdots \subset X(w_{\le r-1}) \subset X(w_{\le r}) = X(w),
\end{align*}
where $e \in W$ is the identity element. Consider the valuations $v_{X(w_{\ge \bullet})}, v_{X(w_{\le \bullet})}$ associated with these sequences. 

Denote by $e_i, f_i, h_i \in \mathfrak{g}$, $i \in I$, the Chevalley generators, by $\{\alpha_i \mid i \in I\}$ the set of simple roots, and by $U^-$ the unipotent radical of the opposite Borel subgroup. Let ${\bf B}^{\rm up} = \{\Xi ^{\rm up}(b) \mid b \in \mathcal{B}(\infty)\}$ be a perfect basis of $\mathbb{C}[U^-]$ (see Definition \ref{definition of perfect bases}), and assume that this basis has the following positivity properties:
\begin{enumerate}
\item[{\rm (i)}] the element $(- f_i) \cdot \Xi^{\rm up} (b)$ belongs to $\sum_{b^\prime \in \mathcal{B}(\infty)} \mathbb{R}_{\ge 0} \Xi^{\rm up} (b^\prime)$ for all $b \in \mathcal{B}(\infty)$ and $i \in I$;
\item[{\rm (ii)}] the product $\Xi^{\rm up}(b) \cdot \Xi^{\rm up}(b^\prime)$ belongs to $\sum_{b'' \in \mathcal{B}(\infty)} \mathbb{R}_{\ge 0} \Xi^{\rm up}(b'')$ for all $b, b^\prime \in \mathcal{B}(\infty)$ such that $\wt(b) \in \{-\alpha_i \mid i \in I\}$.
\end{enumerate}
The existence of a perfect basis with the positivity properties (i) and (ii) follows from a categorification of the negative part of the quantized enveloping algebra (see Proposition \ref{existence with positivity}). Remark that this basis induces a $\mathbb{C}$-basis $\{\Xi_{\lambda, w} ^{\rm up} (b) \mid b \in \mathcal{B}_w (\lambda)\}$ of the space $H^0(X(w), \mathcal{L}_\lambda)$ of global sections (see Section 3 and Proposition \ref{p:pos_D}). Let $\tau_\lambda \in H^0(X(w), \mathcal{L}_\lambda)$ be the restriction of the lowest weight vector in $H^0(G/B, \mathcal{L}_\lambda)$. Write ${\bf a}^{\rm op} \coloneqq (a_r, \ldots, a_1)$ for an element ${\bf a} = (a_1, \ldots, a_r) \in \mathbb{R}^r$, and $H^{\rm op} \coloneqq \{{\bf a}^{\rm op} \mid {\bf a} \in H\}$ for a subset $H \subset \mathbb{R}^r$. The following is the main result of this paper.

\vspace{2mm}\begin{theor}
Let $\lambda$ be a dominant integral weight, ${\bf i} \in I^r$ a reduced word for $w \in W$, and $b \in \mathcal{B}_w(\lambda)$.
\begin{enumerate}
\item[{\rm (1)}] The value $v_{X(w_{\ge \bullet})} (\Xi_{\lambda, w} ^{\rm up} (b)/\tau_\lambda)^{\rm op}$ is equal to the Kashiwara embedding of $b$.
\item[{\rm (2)}] The value $v_{X(w_{\le \bullet})} (\Xi_{\lambda, w} ^{\rm up} (b)/\tau_\lambda)^{\rm op}$ is equal to the Littelmann string parametrization of $b$.
\end{enumerate}
\end{theor}\vspace{2mm}

\begin{coro}
Let $\lambda$ be a dominant integral weight, and ${\bf i} \in I^r$ a reduced word for $w \in W$.
\begin{enumerate}
\item[{\rm (1)}] The Newton-Okounkov body $\Delta(X(w), \mathcal{L}_\lambda, v_{X(w_{\ge \bullet})})^{\rm op}$ is identical to the Nakashima-Zelevinsky polyhedral realization of $\mathcal{B}_w(\lambda)$.
\item[{\rm (2)}] The Newton-Okounkov body $\Delta(X(w), \mathcal{L}_\lambda, v_{X(w_{\le \bullet})})^{\rm op}$ is identical to the Littelmann string polytope for $\mathcal{B}_w(\lambda)$.
\end{enumerate}
\end{coro}\vspace{2mm}

For simplicity, we deal with only finite type case, but our results (Theorem and Corollary above) can be extended to symmetrizable Kac-Moody case without much difficulty. Note that in the case $\mathfrak{g}$ is infinite dimensional, there is no $w \in W$ such that $X(w) = G/B$. Indeed, the full flag variety $G/B$ is infinite dimensional while the Schubert variety $X(w)$ is finite dimensional. Hence in this case, we cannot replace $X(w)$ in Corollary above with $G/B$. See \cite{Kum} for the precise treatment. 

Finally, we should mention some previous works. The computation of the Newton-Okounkov body with respect to the valuation $v_{X(w_{\le \bullet})}$ was partially done by Okounkov \cite{Oko2}. In the case that $G = Sp_{2n}(\mathbb{C})$ and ${\bf i}$ is a specific reduced word for the longest element, he proved that the Newton-Okounkov body with respect to $v_{X(w_{\le \bullet})}$ is identical (after an explicit affine transformation) to the type $C$ Gelfand-Zetlin polytope, which coincides (after an explicit affine transformation) with the corresponding Littelmann string polytope by \cite[Corollary 7]{Lit}. Since the collection $u_1, \ldots, u_r$ of rational functions used in \cite{Oko2} is different from ours, the Newton-Okounkov body computed in \cite{Oko2} is not identical to ours, but they are unimodular equivalent. Note that our approach in this paper is quite different from his.

In the paper \cite{Kir}, Kiritchenko considered the valuation associated with the sequence of translated Schubert varieties: \[w X(e) = w_{\le r} X(e) \subset w_{\le r-1} X(w_{\ge r}) \subset w_{\le r-2} X(w_{\ge r-1}) \subset \cdots \subset w_{\le 1} X(w_{\ge 2}) \subset e X(w_{\ge 1}) = X(w).\] In the case that $G = SL_n(\mathbb{C})$ and ${\bf i}$ is a specific reduced word for the longest element, she proved that the corresponding Newton-Okounkov body is identical to the Feigin-Fourier-Littelmann-Vinberg polytope, which is defined by using Dyck paths (cf.\ \cite{FeFL}). Note that this Newton-Okounkov body is not unimodularly equivalent to the ones with respect to the valuations $v_{X(w_{\ge \bullet})}, v_{X(w_{\le \bullet})}$ in general.

This paper is organized as follows. In Section 2, we recall the definition of Newton-Okounkov bodies. In Section 3, we recall some properties of perfect bases, and review the main results of \cite{FN} and \cite{Kav}. Section 4 is devoted to explaining properties of perfect bases with the positivity properties (i) and (ii). Finally, we prove Theorem above in Section 5.

\vspace{2mm}\begin{ack}\normalfont
The first author is greatly indebted to his supervisor Satoshi Naito for fruitful discussions and helpful suggestions. The second author would like to thank his supervisor Yoshihisa Saito for his support and encouragement. The authors wish to express their gratitude to Yoshiyuki Kimura for pointing out some nontrivial gaps. They would like to thank Xin Fang, Megumi Harada, Kiumars Kaveh, and Valentina Kiritchenko for many comments and suggestions. The second author thanks the University of Caen Normandy, where a part of this paper was written, and Bernard Leclerc for hospitality.
\end{ack}

\section{Newton-Okounkov polytopes of Schubert varieties}

Here we recall the definition of Newton-Okounkov bodies of Schubert varieties, following \cite{HK}, \cite{Kav}, \cite{KK1}, and \cite{KK2}. Let $R$ be a $\mathbb{C}$-algebra without nonzero zero-divisors, and fix a total order $<$ on $\mathbb{Z}^r$, $r \in \mathbb{Z}_{>0}$, respecting the addition. 

\vspace{2mm}\begin{defi}\normalfont\label{def,val}
A map $v \colon R \setminus \{0\} \rightarrow \mathbb{Z}^r$ is called a {\it valuation} on $R$ if the following hold: for every $\sigma, \tau \in R \setminus \{0\}$ and $c \in \mathbb{C} \setminus \{0\}$,
\begin{enumerate}
\item[{\rm (i)}] $v(\sigma \cdot \tau) = v(\sigma) + v(\tau)$,
\item[{\rm (ii)}] $v(c \cdot \sigma) = v(\sigma)$, 
\item[{\rm (iii)}] $v (\sigma + \tau) \ge {\rm min} \{v(\sigma), v(\tau) \}$ unless $\sigma + \tau = 0$. 
\end{enumerate}
\end{defi}\vspace{2mm}

The following is a fundamental property of valuations.

\vspace{2mm}\begin{prop}[{See, for instance, \cite[Proposition 1.8 (2)]{Kav}}]\label{prop1,val}
Let $v$ be a valuation on $R$. For $\sigma_1, \ldots, \sigma_s \in R \setminus \{0\}$, assume that $v(\sigma_1), \ldots, v(\sigma_s)$ are all distinct. Then for $c_1, \ldots, c_s \in \mathbb{C}$ such that $\sigma \coloneqq c_1 \sigma_1 + \cdots + c_s \sigma_s \neq 0$, the following equality holds: \[v(\sigma) = \min\{v(\sigma_t ) \mid 1 \le t \le s,\ c_t \neq 0 \}.\]
\end{prop}\vspace{2mm}

Let $G$ be a connected, simply-connected semisimple algebraic group over $\mathbb{C}$, $\mathfrak{g}$ its Lie algebra, $W$ the Weyl group, and $I$ an index set for the vertices of the Dynkin diagram. Choose a Borel subgroup $B \subset G$ and a maximal torus $T \subset B$. Denote by $\mathfrak{t}$ the Lie algebra of $T$, by $\mathfrak{t}^\ast \coloneqq {\rm Hom}_\mathbb{C} (\mathfrak{t}, \mathbb{C})$ its dual space, and by $\langle \cdot, \cdot \rangle \colon \mathfrak{t}^\ast \times \mathfrak{t} \rightarrow \mathbb{C}$ the canonical pairing. Let $\{\alpha_i \mid i \in I\} \subset \mathfrak{t}^\ast$ be the set of simple roots, $\{h_i \mid i \in I\} \subset \mathfrak{t}$ the set of simple coroots, and $e_i, f_i, h_i \in \mathfrak{g}$, $i \in I$, the Chevalley generators. For $i\in I$, denote by $\mathfrak{g}_i$ the Lie subalgebra of $\mathfrak{g}$ generated by $e_i, f_i, h_i$, which is isomorphic to $\mathfrak{sl}_2(\mathbb{C})$ as a Lie algebra. 

\vspace{2mm}\begin{defi}\normalfont
Let us denote by $X(w)$ for $w \in W$ the Zariski closure of $B \widetilde{w} B/B$ in $G/B$, where $\widetilde{w} \in G$ denotes a lift for $w$; note that the closed subvariety $X(w)$ is independent of the choice of a lift $\widetilde{w}$. The $X(w)$ is called the {\it Schubert variety} corresponding to $w \in W$.
\end{defi}\vspace{2mm}

It is well-known that the Schubert variety $X(w)$ is a normal projective variety of complex dimension $\ell(w)$, where $\ell(w)$ is the length of $w$. Given a dominant integral weight $\lambda$, we define a line bundle $\mathcal{L}_\lambda$ on $G/B$ by \[\mathcal{L}_\lambda \coloneqq (G \times \mathbb{C})/B,\] where $B$ acts on $G \times \mathbb{C}$ on the right as follows: \[(g, c) \cdot b = (g b, \lambda(b) c)\] for $g \in G$, $c \in \mathbb{C}$, and $b \in B$. By restricting this bundle, we obtain a line bundle on $X(w)$, which we denote by the same symbol $\mathcal{L}_\lambda$. Let $V(\lambda)$ be the irreducible highest weight $G$-module with highest weight $\lambda$, $v_\lambda \in V(\lambda)$ the highest weight vector, and $v_{w\lambda} \in V(\lambda)$ the extremal weight vector of weight $w\lambda$ for $w \in W$. Then the {\it Demazure module} $V_w(\lambda)$ corresponding to $w \in W$ is the $B$-submodule of $V(\lambda)$ given by \[V_w(\lambda) \coloneqq \sum_{b \in B} \mathbb{C} b v_{w\lambda}.\] From the Borel-Weil type theorem, we know that the space $H^0(G/B, \mathcal{L}_\lambda)$ (resp., $H^0(X(w), \mathcal{L}_\lambda)$) of global sections is a $G$-module (resp., a $B$-module) isomorphic to the dual module $V(\lambda)^\ast \coloneqq {\rm Hom}_\mathbb{C}(V(\lambda), \mathbb{C})$ (resp., $V_w (\lambda)^\ast \coloneqq {\rm Hom}_\mathbb{C}(V_w(\lambda), \mathbb{C})$). 

\vspace{2mm}\begin{defi}\label{d:order}\normalfont
Define two lexicographic orders $<$ and $\prec$ on $\mathbb{Z}^r$, $r \in \mathbb{Z}_{>0}$, by $(a_1, \ldots, a_r) < (a_1 ^\prime, \ldots, a_r ^\prime)$ (resp., $(a_1, \ldots, a_r) \prec (a_1 ^\prime, \ldots, a_r ^\prime)$) if and only if there exists $1 \le k \le r$ such that $a_1 = a_1 ^\prime, \ldots, a_{k-1} = a_{k-1} ^\prime$, $a_k < a_k ^\prime$ (resp., $a_r = a_r ^\prime, \ldots, a_{k+1} = a_{k+1} ^\prime$, $a_k < a_k ^\prime$). Let $\mathbb{C}(t_1, \ldots, t_r)$ denote the rational function field in $r$ variables. The lexicographic order $<$ on $\mathbb{Z}^r$ induces a total order (denoted by the same symbol $<$) on the set of all monomials in the polynomial ring $\mathbb{C}[t_1, \ldots, t_r]$ as follows: $t_1 ^{a_1} \cdots t_r ^{a_r} < t_1 ^{a_1 ^\prime} \cdots t_r ^{a_r ^\prime}$ if and only if $(a_1, \ldots, a_r) < (a_1 ^\prime, \ldots, a_r ^\prime)$. Let us define two valuations $v^{\rm high}, v^{\rm low} \colon \mathbb{C}(t_1, \ldots, t_r) \setminus \{0\} \rightarrow \mathbb{Z}^r$ by $v^{\rm high}(f/g) \coloneqq v^{\rm high}(f) - v^{\rm high}(g)$, $v^{\rm low}(f/g) \coloneqq v^{\rm low}(f) - v^{\rm low}(g)$ for $f, g \in \mathbb{C}[t_1, \ldots, t_r] \setminus \{0\}$, and by 
\begin{align*}
&v^{\rm high}(f) \coloneqq -(a_1, \ldots, a_r)\ {\rm for}\ f = c t_1 ^{a_1} \cdots t_r ^{a_r} + ({\rm lower\ terms}) \in \mathbb{C}[t_1, \ldots, t_r] \setminus \{0\},\\
&v^{\rm low}(f) \coloneqq (a_1, \ldots, a_r)\ {\rm for}\ f = c t_1 ^{a_1} \cdots t_r ^{a_r} + ({\rm higher\ terms}) \in \mathbb{C}[t_1, \ldots, t_r] \setminus \{0\}, 
\end{align*}
respectively, where $c \in \mathbb{C} \setminus \{0\}$, and by ``lower terms'' (resp., ``higher terms''), we mean a linear combination of monomials smaller (resp., bigger) than $t_1 ^{a_1} \cdots t_r ^{a_r}$ with respect to the total order $<$. Since the total order $<$ on the set of all monomials satisfies $t_1 > \cdots > t_r$, we call the valuation $v^{\rm high}$ (resp., $v^{\rm low}$) on $\mathbb{C}(t_1, \ldots, t_r)$ the {\it highest term valuation} (resp., the {\it lowest term valuation}) with respect to the lexicographic order $t_1 > \cdots > t_r$. Similarly, by using the lexicographic order $\prec$ on $\mathbb{Z}^r$, we define the {\it highest term valuation} $\tilde{v}^{\rm high}$ and the {\it lowest term valuation} $\tilde{v}^{\rm low}$ with respect to the lexicographic order $t_r \succ \cdots \succ t_1$ by
\begin{align*}
&\tilde{v}^{\rm high}(f) \coloneqq -(a_r, \ldots, a_1)\ {\rm for}\ f = c t_1 ^{a_1} \cdots t_r ^{a_r} + ({\rm lower\ terms}) \in \mathbb{C}[t_1, \ldots, t_r] \setminus \{0\},\\
&\tilde{v}^{\rm low}(f) \coloneqq (a_r, \ldots, a_1)\ {\rm for}\ f = c t_1 ^{a_1} \cdots t_r ^{a_r} + ({\rm higher\ terms}) \in \mathbb{C}[t_1, \ldots, t_r] \setminus \{0\}, 
\end{align*}
where $c \in \mathbb{C} \setminus \{0\}$; note that the lexicographic order $\prec$ on $\mathbb{Z}^r$ induces a total order $\prec$ on the set of all monomials satisfying $t_r \succ \cdots \succ t_1$.
\begin{table}[h]
\begin{tabular}{|c||c|c|} \hline
lexicographic order & highest term valuation & lowest term valuation \\ \hline
$t_1 > \cdots > t_r$ & $v^{\rm high}$ & $v^{\rm low}$ \\ \hline
$t_r \succ \cdots \succ t_1$ & $\tilde{v}^{\rm high}$ & $\tilde{v}^{\rm low}$ \\ \hline
\end{tabular}
\end{table}
\end{defi}\vspace{2mm}

\begin{ex}\normalfont
If $r = 3$ and $f = t_1 t_2 + t_3 ^2 \in \mathbb{C}[t_1, t_2, t_3]$, then it follows that $v^{\rm high}(f) = -(1, 1, 0)$, $v^{\rm low}(f) = (0, 0, 2)$, $\tilde{v}^{\rm high}(f) = -(2, 0, 0)$, and $\tilde{v}^{\rm low}(f) = (0, 1, 1)$.
\end{ex}\vspace{2mm}

Let $U^-$ denote the unipotent radical of the opposite Borel subgroup, $U^- _i \subset U^-$ the opposite root subgroup corresponding to an index $i \in I$, and set $\mathfrak{u}^- \coloneqq {\rm Lie}(U^-)$, $\mathfrak{u}^- _i \coloneqq {\rm Lie}(U^- _i) = \mathbb{C} f_i$. We regard $U^-$ as an affine open subset of $G/B$ by: \[U^- \hookrightarrow G/B,\ u \mapsto u \bmod B.\] Consider the set-theoretic intersection $U^- \cap X(w)$ in $G/B$. Since the intersection is an open subset of $X(w)$, it acquires an open subvariety structure from $X(w)$. Remark that this is identical to the closed subvariety structure on $U^- \cap X(w)$ induced from $U^-$, since a reduced subscheme structure on the locally closed subset $U^- \cap X(w) \subset G/B$ is unique. Let ${\bf i} = (i_1, \ldots, i_r) \in I^r$ be a reduced word for $w \in W$. It is well-known that the product map $U_{i_1} ^- \times \cdots \times U_{i_r} ^- \rightarrow U^- \cap X(w)$, $(u_1, \ldots, u_r) \mapsto u_1 \cdots u_r \bmod B$, is a birational morphism (see, for instance, \cite[Part I\hspace{-.1em}I, Chapter 13]{Jan1}); therefore, the function field $\mathbb{C}(X(w)) = \mathbb{C}(U^- \cap X(w))$ is identified with $\mathbb{C}(U_{i_1} ^- \times \cdots \times U_{i_r} ^-)$. By using the isomorphism $\mathbb{C}^r \xrightarrow{\sim} U_{i_1} ^- \times \cdots \times U_{i_r} ^-$ of varieties given by $(t_1, \ldots, t_r) \mapsto (\exp(t_1 f_{i_1}), \ldots, \exp(t_r f_{i_r}))$, we identify the function field $\mathbb{C}(X(w)) = \mathbb{C}(U_{i_1} ^- \times \cdots \times U_{i_r} ^-)$ with the rational function field $\mathbb{C}(t_1, \ldots, t_r)$. Now we define valuations $v_{\bf i} ^{\rm high}, v_{\bf i} ^{\rm low}, \tilde{v}_{\bf i} ^{\rm high}, \tilde{v}_{\bf i} ^{\rm low}$ on $\mathbb{C}(X(w))$ to be $v^{\rm high}, v^{\rm low}, \tilde{v}^{\rm high}, \tilde{v}^{\rm low}$ on $\mathbb{C}(t_1, \ldots, t_r)$, respectively. If we set $w_{\ge k} \coloneqq s_{i_k} s_{i_{k+1}} \cdots s_{i_r}$ and $w_{\le k} \coloneqq s_{i_1} s_{i_2} \cdots s_{i_k}$ for $1 \le k \le r$, then we obtain two sequences of subvarieties of $X(w)$:
\begin{align*}
&X(w_{\ge \bullet}) \colon X(e) \subset X(w_{\ge r}) \subset X(w_{\ge r-1}) \subset \cdots \subset X(w_{\ge 2}) \subset X(w_{\ge 1}) = X(w)\ {\rm and}\\
&X(w_{\le \bullet}) \colon X(e) \subset X(w_{\le 1}) \subset X(w_{\le 2}) \subset \cdots \subset X(w_{\le r-1}) \subset X(w_{\le r}) = X(w),
\end{align*}
where $e \in W$ is the identity element. As discussed in Introduction, we construct two valuations $v_{X(w_{\ge \bullet})}$ and $v_{X(w_{\le \bullet})}$ out of these sequences; note that $X(w_{\ge k})$ (resp., $X(w_{\le k})$) is a normal subvariety of $X(w_{\ge k-1})$ (resp., $X(w_{\le k+1})$) for each $k$. Now it follows immediately that \[v_{\bf i} ^{\rm low} = v_{X(w_{\ge \bullet})}\ {\rm and}\ \tilde{v}_{\bf i} ^{\rm low} = v_{X(w_{\le \bullet})}.\] Consider the left action of $U_{i_1} ^-$ (resp., the right action of $U_{i_r} ^-$) on $U_{i_1} ^- \times \cdots \times U_{i_r} ^-$ given by 
\begin{align*}
u \cdot (u_1, \ldots, u_r) \coloneqq (u u_1, \ldots, u_r)\quad ({\rm resp.,}\ (u_1, \ldots, u_r) \cdot u^\prime \coloneqq (u_1, \ldots, u_r u^\prime))
\end{align*}
for $u_1 \in U_{i_1} ^-, \ldots, u_r \in U_{i_r} ^-$, and $u \in U_{i_1} ^-$ (resp., $u^\prime \in U_{i_r} ^-$); this induces a left action of $\mathfrak{u}_{i_1} ^-$ (resp., a right action of $\mathfrak{u}_{i_r} ^-$) on $\mathbb{C}[t_1, \ldots, t_r] = \mathbb{C}[U_{i_1} ^- \times \cdots \times U_{i_r} ^-]$, which is given by:
\begin{align}\label{derivation}
f_{i_1} \cdot f(t_1, \ldots, t_r) &= - \frac{\partial}{\partial t_1} f(t_1, \ldots, t_r)\\
({\rm resp.,}\ f(t_1, \ldots, t_r) \cdot f_{i_r} &= - \frac{\partial}{\partial t_r} f(t_1, \ldots, t_r))
\end{align}
for $f(t_1, \ldots, t_r) \in \mathbb{C}[t_1, \ldots, t_r]$ (see \cite[\S\S 3.2]{FN} and \cite[Proposition 2.2]{Kav}). 

\vspace{2mm}\begin{prop}[{See \cite[Proposition 3.9]{FN} and \cite[Proof of Theorem 4.1]{Kav}}]\label{valuation Chevalley}
Let $f(t_1, \ldots, t_r) \in \mathbb{C}[t_1, \ldots, t_r]$ be a nonzero polynomial.
\begin{enumerate}
\item[{\rm (1)}] Write $v_{\bf i} ^{\rm high}(f(t_1, \ldots, t_r)) = -(a_1, \ldots, a_r)$. Then the following equalities hold:
\begin{align*}
&a_1 = \max\{a \in \mathbb{Z}_{\ge 0} \mid f_{i_1} ^a \cdot f(t_1, \ldots, t_r) \neq 0\},\\
&a_2 = \max\{a \in \mathbb{Z}_{\ge 0} \mid f_{i_2} ^a \cdot (f_{i_1} ^{a_1} \cdot f(t_1, \ldots, t_r))|_{X(w_{\ge 2})} \neq 0\},\\
&\ \vdots\\
&a_r = \max\{a \in \mathbb{Z}_{\ge 0} \mid f_{i_r} ^a \cdot (\cdots (f_{i_2} ^{a_2} \cdot (f_{i_1} ^{a_1} \cdot f(t_1, \ldots, t_r))|_{X(w_{\ge 2})}) \cdots)|_{X(w_{\ge r})} \neq 0\}.
\end{align*}
\item[{\rm (2)}] Write $\tilde{v}_{\bf i} ^{\rm high}(f(t_1, \ldots, t_r)) = -(a_r ^\prime, \ldots, a_1 ^\prime)$. Then the following equalities hold:
\begin{align*}
&a_r ^\prime = \max\{a \in \mathbb{Z}_{\ge 0} \mid f(t_1, \ldots, t_r) \cdot f_{i_r} ^a \neq 0\},\\
&a_{r-1} ^\prime = \max\{a \in \mathbb{Z}_{\ge 0} \mid (f(t_1, \ldots, t_r) \cdot f_{i_r} ^{a_r ^\prime})|_{X(w_{\le r-1})} \cdot f_{i_{r-1}} ^a \neq 0\},\\
&\ \vdots\\
&a_1 ^\prime = \max\{a \in \mathbb{Z}_{\ge 0} \mid (\cdots ((f(t_1, \ldots, t_r) \cdot f_{i_r} ^{a_r ^\prime})|_{X(w_{\le r-1})} \cdot f_{i_{r-1}} ^{a_{r-1} ^\prime}) \cdots)|_{X(w_{\le 1})} \cdot f_{i_1} ^a \neq 0\}.
\end{align*}
\end{enumerate}
\end{prop}

\vspace{2mm}\begin{defi}\normalfont\label{Newton-Okounkov convex body}
For a dominant integral weight $\lambda$ and a reduced word ${\bf i} = (i_1, \ldots, i_r) \in I^r$ for $w \in W$, take $v_{\bf i} \in \{v_{\bf i} ^{\rm high}, v_{\bf i} ^{\rm low}, \tilde{v}_{\bf i} ^{\rm high}, \tilde{v}_{\bf i} ^{\rm low}\}$ and $\tau \in H^0(X(w), \mathcal{L}_\lambda) \setminus \{0\}$. Define a subset $S(X(w), \mathcal{L}_\lambda, v_{\bf i}, \tau) \subset \mathbb{Z}_{>0} \times \mathbb{Z}^r$ by \[S(X(w), \mathcal{L}_\lambda, v_{\bf i}, \tau) \coloneqq \bigcup_{k>0} \{(k, v_{\bf i}(\sigma/\tau^k)) \mid \sigma \in H^0(X(w), \mathcal{L}_\lambda ^{\otimes k}) \setminus \{0\}\},\] and denote by $C(X(w), \mathcal{L}_\lambda, v_{\bf i}, \tau) \subset \mathbb{R}_{\ge 0} \times \mathbb{R}^r$ the smallest real closed cone containing $S(X(w), \mathcal{L}_\lambda, v_{\bf i}, \tau)$. Since $v_{\bf i}$ is a valuation, it follows that the subset $S(X(w), \mathcal{L}_\lambda, v_{\bf i}, \tau)$ is a semigroup, and that the real closed cone $C(X(w), \mathcal{L}_\lambda, v_{\bf i}, \tau)$ is convex. Let us define a subset $\Delta(X(w), \mathcal{L}_\lambda, v_{\bf i}, \tau) \subset \mathbb{R}^r$ by \[\Delta(X(w), \mathcal{L}_\lambda, v_{\bf i}, \tau) \coloneqq \{{\bf a} \in \mathbb{R}^r \mid (1, {\bf a}) \in C(X(w), \mathcal{L}_\lambda, v_{\bf i}, \tau)\};\]
this is called the {\it Newton-Okounkov body} of $X(w)$ associated with $\mathcal{L}_\lambda$, $v_{\bf i}$, and $\tau$.
\end{defi}\vspace{2mm}

As we will see in Sections 3 and 5, the Newton-Okounkov body $\Delta(X(w), \mathcal{L}_\lambda, v_{\bf i}, \tau)$ is indeed a rational convex polytope. Hence it is also called a {\it Newton-Okounkov polytope}.

\vspace{2mm}\begin{rem}\normalfont\label{independence}
If we take another section $\tau^\prime \in H^0 (X(w), \mathcal{L}_\lambda) \setminus \{0\}$, then $S(X(w), \mathcal{L}_\lambda, v_{\bf i}, \tau^\prime)$ is the shift of $S(X(w), \mathcal{L}_\lambda, v_{\bf i}, \tau)$ by $k v_{\bf i}(\tau/\tau^\prime)$ in $\{k\} \times \mathbb{Z}^r$ for $k \in \mathbb{Z}_{>0}$, that is, \[S(X(w), \mathcal{L}_\lambda, v_{\bf i}, \tau^\prime) \cap (\{k\} \times \mathbb{Z}^r) = S(X(w), \mathcal{L}_\lambda, v_{\bf i}, \tau) \cap (\{k\} \times \mathbb{Z}^r) + (0, k v_{\bf i}(\tau/\tau^\prime)).\] Hence it follows that $\Delta(X(w), \mathcal{L}_\lambda, v_{\bf i}, \tau^\prime) = \Delta(X(w), \mathcal{L}_\lambda, v_{\bf i}, \tau) + v_{\bf i}(\tau/\tau^\prime)$. Thus, the Newton-Okounkov body $\Delta(X(w), \mathcal{L}_\lambda, v_{\bf i}, \tau)$ does not essentially depend on the choice of $\tau \in H^0 (X(w), \mathcal{L}_\lambda) \setminus \{0\}$; hence it is also denoted simply by $\Delta(X(w), \mathcal{L}_\lambda, v_{\bf i})$.
\end{rem}\vspace{2mm}

\begin{ex}\normalfont\label{main example}
Let $G = SL_3(\mathbb{C})$ (of type $A_2$), $I = \{1, 2\}$, ${\bf i} = (1, 2, 1)$ a reduced word for the longest element $w_0$ of $W$, and $\lambda = \alpha_1 + \alpha_2$. Then the Schubert variety $X(w_0)$ is identical to the full flag variety $G/B$. Recall that the coordinate ring $\mathbb{C}[U^-]$ is regarded as a $\mathbb{C}$-subalgebra of the polynomial ring $\mathbb{C}[t_1, t_2, t_3]$ by using the birational morphism
\[\mathbb{C}^3 \rightarrow U^-,\ (t_1, t_2, t_3) \mapsto \exp(t_1 f_1)\exp(t_2 f_2)\exp(t_3 f_1).\] Since we have \[\exp(t_1 f_1)\exp(t_2 f_2)\exp(t_3 f_1) =
\begin{pmatrix}
1 & 0 & 0\\
t_1 + t_3 & 1 & 0 \\
t_2 t_3 & t_2 &1
\end{pmatrix},\] the coordinate ring $\mathbb{C}[U^-]$ is identical to the $\mathbb{C}$-subalgebra $\mathbb{C}[t_1 + t_3, t_2, t_2 t_3]$ of $\mathbb{C}[t_1, t_2, t_3]$. In addition, by standard monomial theory (see, for instance, \cite[Section 2]{Ses}), we deduce that for a specific section $\tau_\lambda \in H^0(G/B, \mathcal{L}_\lambda)$, the $\mathbb{C}$-subspace $\{\sigma/\tau_\lambda \mid \sigma \in H^0(G/B, \mathcal{L}_\lambda)\}$ of $\mathbb{C}(U^-)$ is spanned by \[\{1, t_1 + t_3, t_2, t_1 t_2, t_2 t_3, t_1 t_2 (t_1 + t_3), t_2 ^2 t_3, t_1 t_2 ^2 t_3\}.\] Now we obtain the following list.

\begin{table}[h]
\begin{tabular}{|c||c|c|c|c|c|c|c|c|} \hline
Valuation & $1$ & $t_1 + t_3$ & $t_2$ & $t_1 t_2$ & $t_2 t_3$ & $t_1 t_2 (t_1 + t_3)$ & $t_2 ^2 t_3$ & $t_1 t_2 ^2 t_3$\\ \hline
$v_{\bf i} ^{\rm high}$ & $(0, 0, 0)$ & $-(1, 0, 0)$ & $-(0, 1, 0)$ & $-(1, 1, 0)$ & $-(0, 1, 1)$ & $-(2, 1, 0)$ & $-(0, 2, 1)$ & $-(1, 2, 1)$\\ \hline
$v_{\bf i} ^{\rm low}$ & $(0, 0, 0)$ & $\ \ (0, 0, 1)$ & $\ \ (0, 1, 0)$ & $\ \ (1, 1, 0)$ & $\ \ (0, 1, 1)$ & $\ \ (1, 1, 1)$ & $\ \ (0, 2, 1)$ & $\ \ (1, 2, 1)$\\ \hline
$\tilde{v}_{\bf i} ^{\rm high}$ & $(0, 0, 0)$ & $-(1, 0, 0)$ & $-(0, 1, 0)$ & $-(0, 1, 1)$ & $-(1, 1, 0)$ & $-(1, 1, 1)$ & $-(1, 2, 0)$ & $-(1, 2, 1)$\\ \hline
$\tilde{v}_{\bf i} ^{\rm low}$ & $(0, 0, 0)$ & $\ \ (0, 0, 1)$ & $\ \ (0, 1, 0)$ & $\ \ (0, 1, 1)$ & $\ \ (1, 1, 0)$ & $\ \ (0, 1, 2)$ & $\ \ (1, 2, 0)$ & $\ \ (1, 2, 1)$\\ \hline
\end{tabular}
\end{table}

\noindent For $v_{\bf i} \in \{v_{\bf i} ^{\rm high}, v_{\bf i} ^{\rm low}, \tilde{v}_{\bf i} ^{\rm high}, \tilde{v}_{\bf i} ^{\rm low}\}$, the Newton-Okounkov body $\Delta(G/B, \mathcal{L}_\lambda, v_{\bf i}, \tau_\lambda)$ is identical to the convex hull of the corresponding eight points in the list above; see the figures 1--4.

\begin{figure}[h]
\begin{minipage}[b]{0.45\linewidth}
\begin{center}
\includegraphics[width=2.0cm]{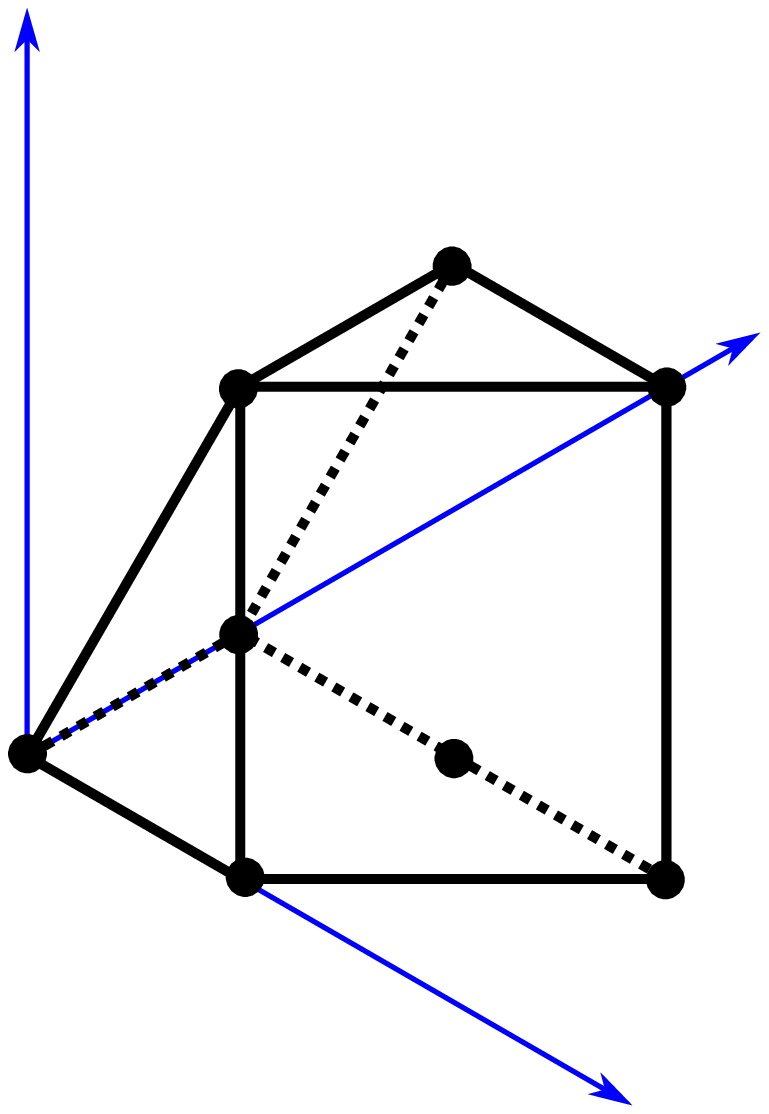}
\caption{$-\Delta(G/B, \mathcal{L}_\lambda, v_{\bf i} ^{\rm high}, \tau_\lambda)$}
\end{center}
\end{minipage}
\begin{minipage}[b]{0.45\linewidth}
\begin{center}
\includegraphics[width=2.0cm]{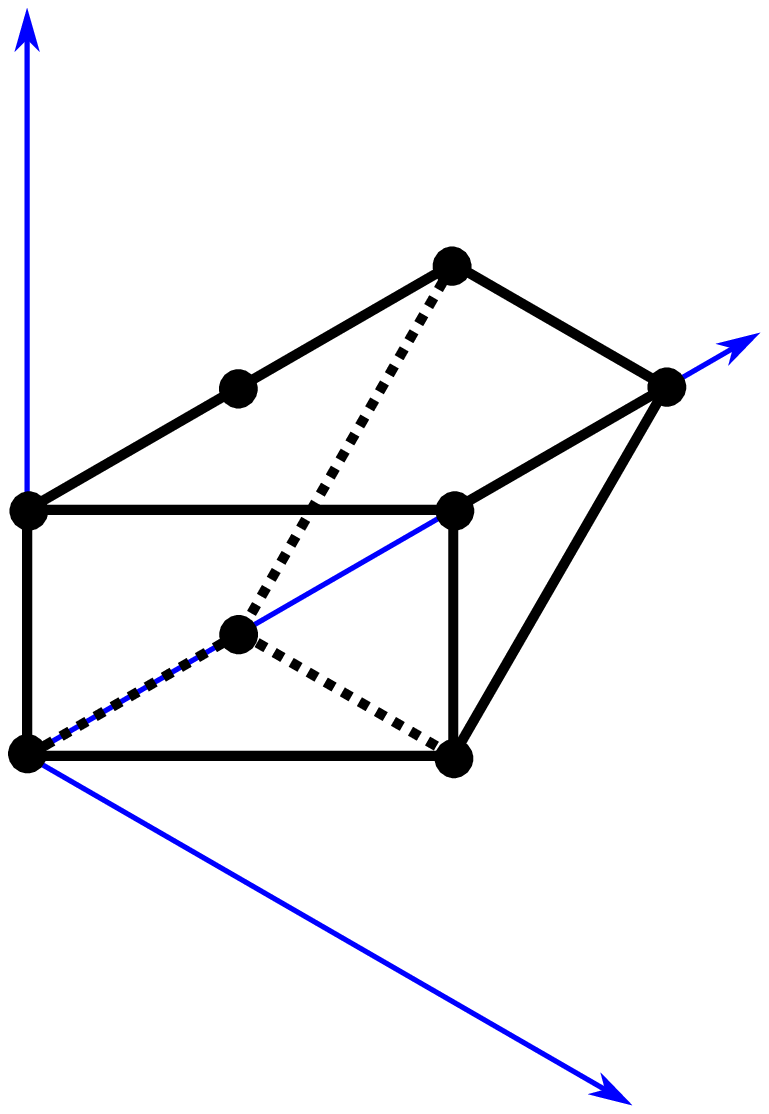}
\caption{$\Delta(G/B, \mathcal{L}_\lambda, v_{\bf i} ^{\rm low}, \tau_\lambda)$}
\end{center}
\end{minipage}\\\vspace{4mm}
\begin{minipage}[b]{0.45\linewidth}
\begin{center}
\includegraphics[width=2.0cm]{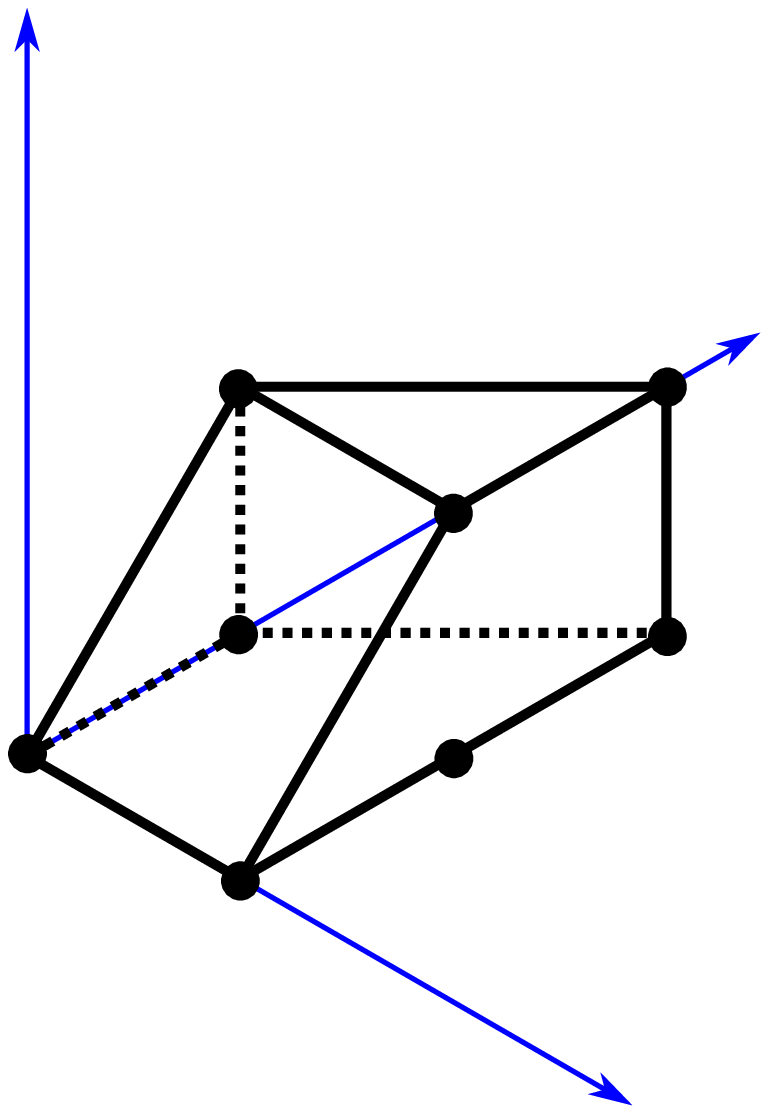}
\caption{$-\Delta(G/B, \mathcal{L}_\lambda, \tilde{v}_{\bf i} ^{\rm high}, \tau_\lambda)$}
\end{center}
\end{minipage}
\begin{minipage}[b]{0.45\linewidth}
\begin{center}
\includegraphics[width=2.0cm]{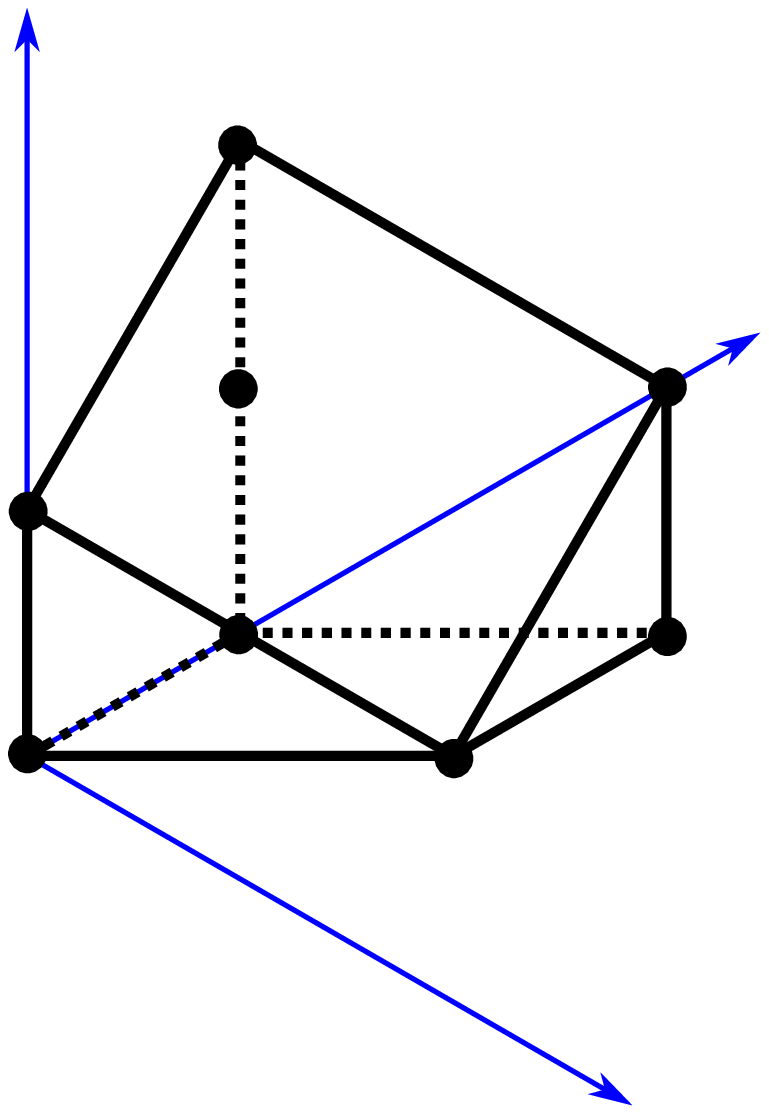}
\caption{$\Delta(G/B, \mathcal{L}_\lambda, \tilde{v}_{\bf i} ^{\rm low}, \tau_\lambda)$}
\end{center}
\end{minipage}
\end{figure}

\noindent Hence we deduce that
\begin{align*}
&\Delta(G/B, \mathcal{L}_\lambda, v_{\bf i} ^{\rm low}, \tau_\lambda) = -\Delta(G/B, \mathcal{L}_\lambda, \tilde{v}_{\bf i} ^{\rm high}, \tau_\lambda)^{\rm op},\ {\rm and}\\
&\Delta(G/B, \mathcal{L}_\lambda, \tilde{v}_{\bf i} ^{\rm low}, \tau_\lambda) = -\Delta(G/B, \mathcal{L}_\lambda, v_{\bf i} ^{\rm high}, \tau_\lambda)^{\rm op},
\end{align*}
where we write $H^{\rm op} \coloneqq \{(a_3, a_2, a_1) \mid (a_1, a_2, a_3) \in H\}$ for a subset $H \subset \mathbb{R}^3$. Our main result (Corollary \ref{corollary of main result}) states that these coincidences of Newton-Okounkov bodies hold also for arbitrary $G$, ${\bf i}$, and $\lambda$; only restriction is that we need to take a specific section $\tau_\lambda$.
\end{ex}

\section{Kashiwara crystal bases and perfect bases}

In this section, we first review the definitions and properties of perfect bases of the space $H^0(G/B, \mathcal{L}_\lambda)\simeq V(\lambda)^{\ast}$ of global sections and the coordinate ring $\mathbb{C}[U^-]$. They are convenient tools for calculating the Newton-Okounkov bodies of Schubert varieties. Next we review the main results of \cite{FN} and \cite{Kav}. 

Let $P \subset \mathfrak{t}^\ast$ be the weight lattice of $\mathfrak{g}$, and $P_+ \subset P$ the set of dominant integral weights. For $\lambda\in P_+$ and $\mu\in P$, set 
\[
V(\lambda)_{\mu} \coloneqq \{v\in V(\lambda) \mid h \cdot v = \langle \mu, h\rangle v\ \text{for all } h\in\mathfrak{t}\}. 
\]
The action of $\mathfrak{g}$ on the dual space $V(\lambda)^{\ast}$, $\lambda\in P_+$, is given by $\langle x\cdot f, v\rangle:=-\langle f, x\cdot v\rangle$ for $f\in V(\lambda)^{\ast}$, $x\in\mathfrak{g}$, and $v \in V(\lambda)$, where $\langle \cdot, \cdot \rangle \colon V(\lambda)^{\ast} \times V(\lambda) \rightarrow \mathbb{C}$ is the canonical pairing. Since $V(\lambda)=\bigoplus_{\mu\in P}V(\lambda)_{\mu}$, the dual space $V(\lambda)_{\mu}^{\ast} \coloneqq (V(\lambda)_{\mu})^{\ast}$ is regarded as a subspace of $V(\lambda)^{\ast}$. For $i \in I$ and $f \in V(\lambda)^{\ast}\setminus\{0\}$, set 
\[
\varepsilon_i (f)\coloneqq \max\{k\in \mathbb{Z}_{\geq 0}\mid f_i^{k} \cdot f\neq 0\}. 
\]
Let $\varepsilon_i(f)\coloneqq -\infty$ for $f = 0 \in V(\lambda)^{\ast}$. For $i\in I$ and $k\in\mathbb{Z}_{\geq 0}$, set
\[
(V(\lambda)^{\ast})^{<k, i}\coloneqq \{f\in V(\lambda)^{\ast}\mid \varepsilon_i (f)<k\}. 
\]

\vspace{2mm}\begin{defi}[{See \cite[Definition 5.30]{BK} and \cite[Definition 2.5]{KOP1}}]\label{definition of perfect bases for modules}\normalfont
Let $\lambda\in P_+$. A $\mathbb{C}$-basis ${\bf B}^{\rm up}(\lambda) \subset V(\lambda)^{\ast}$ is said to be {\it perfect} if the following conditions hold:
\begin{enumerate}
\item[{\rm (i)}] ${\bf B}^{\rm up}(\lambda) = \coprod_{\mu\in P} {\bf B}^{\rm up}(\lambda)_{\mu}$, where ${\bf B}^{\rm up}(\lambda)_{\mu} \coloneqq {\bf B}^{\rm up}(\lambda) \cap V(\lambda)_{\mu}^{\ast}$,
\item[{\rm (ii)}] ${\bf B}^{\rm up}(\lambda)_{\lambda} = \{\tau_{\lambda}\}$, where $\langle \tau_{\lambda}, v_{\lambda}\rangle=1$,
\item[{\rm (iii)}] for $i \in I$ and $\tau \in {\bf B}^{\rm up}(\lambda)$ with $f_i \cdot \tau \neq 0$, there exists a unique element $\tilde{e}_i (\tau) \in {\bf B}^{\rm up}(\lambda)$ such that 
\[f_i \cdot \tau \in \mathbb{C}^\times \tilde{e}_i(\tau)+ (V(\lambda)^{\ast})^{<\varepsilon_i(\tau)-1, i},\] 
where $\mathbb{C}^\times \coloneqq \mathbb{C} \setminus \{0\}$,
\item[{\rm (iv)}] if $\tilde{e}_i(\tau)=\tilde{e}_i(\tau')$ for $\tau, \tau' \in {\bf B}^{\rm up}(\lambda)$ and some $i \in I$, then we have $\tau = \tau'$. 
\end{enumerate} 
\end{defi}\vspace{2mm}

Next we review the definition of a perfect basis of $\mathbb{C}[U^-]$. Let $U(\mathfrak{u}^-)$ be the universal enveloping algebra of $\mathfrak{u}^-$. The algebra $U(\mathfrak{u}^-)$ has a Hopf algebra structure given by the following coproduct $\Delta$, counit $\varepsilon$, and antipode $S$:
\[\Delta(f_i) = f_i \otimes 1 + 1 \otimes f_i,\ \varepsilon(f_i) = 0,\ {\rm and}\ S(f_i) = -f_i\]
for $i \in I$. We can regard $U(\mathfrak{u}^-)$ as a multigraded $\mathbb{C}$-algebra: \[U(\mathfrak{u}^-) = \bigoplus_{{\bf d} \in \mathbb{Z}^I _{\ge 0}} U(\mathfrak{u}^-)_{\bf d},\] where the homogeneous component $U(\mathfrak{u}^-)_{\bf d}$ for ${\bf d} = (d_i)_{i \in I} \in \mathbb{Z}^I _{\ge 0}$ is defined to be the $\mathbb{C}$-subspace of $U(\mathfrak{u}^-)$ spanned by elements $f_{j_1} \cdots f_{j_{|{\bf d}|}}$ for which the cardinality of $\{1 \le k \le |{\bf d}| \mid j_k = i\}$ is equal to $d_i$ for all $i \in I$; here we set $|{\bf d}| \coloneqq \sum_{i \in I} d_i$. Let \[U(\mathfrak{u}^-)^\ast _{\rm gr} = \bigoplus_{{\bf d} \in \mathbb{Z}^I _{\ge 0}}U(\mathfrak{u}^-)^\ast _{{\rm gr}, {\bf d}} \coloneqq \bigoplus_{{\bf d} \in \mathbb{Z}^I _{\ge 0}} {\rm Hom}_\mathbb{C} (U(\mathfrak{u}^-)_{\bf d}, \mathbb{C})\] be the graded dual of $U(\mathfrak{u}^-)$ endowed with the dual Hopf algebra structure. Note that the coordinate ring $\mathbb{C}[U^-]$ also has a Hopf algebra structure given by the following coproduct $\Delta$, counit $\varepsilon$, and antipode $S$: \[\Delta(f) ((u_1, u_2)) = f(u_1 u_2),\ \varepsilon(f) = f(e),\ {\rm and}\ S(f) (u) = f(u^{-1})\] for $f \in \mathbb{C}[U^-]$ and $u, u_1, u_2 \in U^-$, where $e \in U^-$ is the identity element. It is well-known that this Hopf algebra $\mathbb{C}[U^-]$ is isomorphic to the dual Hopf algebra $U(\mathfrak{u}^-)_{\rm gr} ^\ast$ (see, for instance, \cite[Proposition 5.1]{GLS}).
Let $\langle \cdot, \cdot \rangle \colon U(\mathfrak{u}^-)^\ast _{\rm gr} \times U(\mathfrak{u}^-) \rightarrow \mathbb{C}$ denote the canonical pairing. Define a $U(\mathfrak{u}^-)$-bimodule structure on $U(\mathfrak{u}^-)^\ast _{\rm gr}$ by 
\begin{align*}
&\langle x \cdot \rho, y \rangle \coloneqq - \langle \rho, x \cdot y \rangle,\ {\rm and}\\
&\langle \rho \cdot x, y \rangle \coloneqq - \langle \rho, y \cdot x \rangle
\end{align*}
for $x \in \mathfrak{u}^-$, $\rho \in U(\mathfrak{u}^-)^\ast _{\rm gr}$, and $y \in U(\mathfrak{u}^-)$. Also, the coordinate ring $\mathbb{C}[U^-]$ has a natural $U^-$-bimodule structure, which is given by 
\begin{align*}
&(u_1 \cdot f)(u_2) \coloneqq f(u_1 ^{-1} u_2),\ {\rm and}\\
&(f \cdot u_1)(u_2) \coloneqq f(u_2 u_1 ^{-1})
\end{align*}
for $u_1, u_2 \in U^-$ and $f \in \mathbb{C}[U^-]$. This induces a $U(\mathfrak{u}^-)$-bimodule structure on $\mathbb{C}[U^-]$. Note that the isomorphism of Hopf algebras $U(\mathfrak{u}^-)_{\rm gr} ^\ast \simeq \mathbb{C}[U^-]$ is compatible with the $U(\mathfrak{u}^-)$-bimodule structures. Henceforth we will identify  $U(\mathfrak{u}^-)_{\rm gr} ^\ast$ with $\mathbb{C}[U^-]$. Define a $\mathbb{C}$-algebra anti-involution $\ast$ on $U(\mathfrak{u}^-)$ by $f_i ^\ast \coloneqq f_i$ for all $i \in I$. This map is a $\mathbb{C}$-coalgebra involution; hence it induces a $\mathbb{C}$-algebra involution on $U(\mathfrak{u}^-)_{\rm gr} ^\ast = \mathbb{C}[U^-]$ (also denoted by $\ast$). For $i\in I$ and $f\in \mathbb{C}[U^-]\setminus\{0\}$, set 
\[
\varepsilon_i (f)\coloneqq \max\{k\in \mathbb{Z}_{\geq 0}\mid f_i^{k} \cdot f\neq 0\}. 
\]
Let $\varepsilon_i(f)\coloneqq -\infty$ for $f = 0 \in \mathbb{C}[U^-]$. For $i\in I$ and $k\in\mathbb{Z}_{\geq 0}$, set
\[
\mathbb{C}[U^-]^{<k, i}\coloneqq \{f\in \mathbb{C}[U^-]\mid \varepsilon_i (f)<k\}. 
\]

\vspace{2mm}\begin{defi}[{See \cite[Definition 5.30]{BK} and \cite[Definition 4.5]{KOP2}}]\label{definition of perfect bases}\normalfont
A $\mathbb{C}$-basis ${\bf B}^{\rm up} \subset \mathbb{C}[U^-]=U(\mathfrak{u}^-)^\ast _{\rm gr}$ is said to be {\it perfect} if the following conditions hold:
\begin{enumerate}
\item[{\rm (i)}] ${\bf B}^{\rm up} = \coprod_{{\bf d} \in \mathbb{Z}^I _{\ge 0}} {\bf B}^{\rm up} _{\bf d}$, where ${\bf B}^{\rm up} _{\bf d} \coloneqq {\bf B}^{\rm up} \cap U(\mathfrak{u}^-)^\ast _{{\rm gr}, {\bf d}}$,
\item[{\rm (ii)}] ${\bf B}^{\rm up} _{(0,\dots, 0)} = \{\tau_{\infty}\}$, where $\langle \tau_{\infty}, 1\rangle=1$,
\item[{\rm (iii)}] for $i \in I$ and $\tau \in {\bf B}^{\rm up}$ with $f_i \cdot \tau \neq 0$, there exists a unique element $\tilde{e}_i(\tau) \in {\bf B}^{\rm up}$ such that 
\[f_i \cdot \tau \in \mathbb{C}^\times \tilde{e}_i(\tau)+ \mathbb{C}[U^-]^{<\varepsilon_i(\tau)-1, i},\] 
\item[{\rm (iv)}] if $\tilde{e}_i(\tau)=\tilde{e}_i(\tau')$ for $\tau, \tau' \in {\bf B}^{\rm up}$ and some $i \in I$, then we have $\tau = \tau'$. 
\end{enumerate} 
Moreover, in this paper, we always impose the following $\ast$-stable condition on a perfect basis:
\begin{enumerate}
\item[{\rm (v)}] $({\bf B}^{\rm up})^{\ast}={\bf B}^{\rm up}$.
\end{enumerate}
\end{defi}\vspace{2mm}

We list some examples of perfect bases here. In particular, Example \ref{e:perfect_KLR} is extremely important in this paper. See also Proposition \ref{existence with positivity}.

\vspace{2mm}\begin{ex}\label{e:upper global}\normalfont
The \emph{upper global bases}  ($=$ the \emph{dual canonical bases}) of $V(\lambda)^{\ast}$, $\lambda\in P_+$, and $\mathbb{C}[U^-]=U(\mathfrak{u}^-)^\ast _{\rm gr}$ are typical examples of perfect bases. They are the dual bases of the \emph{lower global bases} ($=$ the \emph{canonical bases}), introduced by Lusztig \cite{Lus_can, Lus_quivers, Lus1} and Kashiwara \cite{Kas1,Kas2} via quantized enveloping algebras associated with $\mathfrak{g}$. See \cite[Proposition 5.3.1]{Kas3}, \cite[Theorem 2.1.1]{Kas4} (and also \cite[Proposition 2.8]{FN}) for the perfectness. The upper global bases are denoted by $\{G_{\lambda}^{\rm up}(b) \mid b \in \mathcal{B}(\lambda)\} \subset V(\lambda)^\ast$ and $\{G^{\rm up}(b) \mid b \in \mathcal{B}(\infty)\} \subset \mathbb{C}[U^-]$, respectively. 
\end{ex}

\vspace{2mm}\begin{ex}\normalfont
When $\mathfrak{g}$ is simply-laced, Lusztig \cite{Lus2} constructed a specific $\mathbb{C}$-basis of $U(\mathfrak{u}^-)$, called the {\it semicanonical basis}. The dual basis of the semicanonical basis, called the \emph{dual semicanonical basis}, is a perfect basis by \cite[Proof of Lemma 2.4 and Section 3]{Lus2}.
\end{ex}\vspace{2mm}

\begin{ex}\label{e:perfect_KLR}\normalfont
Khovanov-Lauda \cite{KL1, KL2} and Rouquier \cite{Rou} introduced the family $\{R_{\bf d} \mid {\bf d} \in \mathbb{Z}^I _{\ge 0}\}$ of $\mathbb{Z}$-graded algebras, called {\it Khovanov-Lauda-Rouquier algebras} or {\it quiver Hecke algebras}, which categorifies the negative half $U_q(\mathfrak{u}^-)$ of the quantized enveloping algebra. To be precise, the direct sum 
$\bigoplus_{{\bf d} \in \mathbb{Z}^I _{\ge 0}}G_0(R_{\bf d}\text{-}{\rm gmod})$ of the Grothendieck groups of finite dimensional graded $R_{\bf d}$-modules has the algebra structure which comes from the induction functor \cite[Proposition 3.1]{KL1}. Moreover the resulting $\mathbb{Z}[q^{\pm 1}]$-algebra (the action of $q$ is induced from the grading shift functor) is isomorphic to a certain $\mathbb{Z}[q^{\pm 1}]$-form $U_{q, \mathbb{Z}}(\mathfrak{u}^-)$ of $U_q(\mathfrak{u}^-)$, which is isomorphic to $\mathbb{C}[U^-]$ if we apply the functor $\mathbb{C}\otimes_{\mathbb{Z}[q^{\pm 1}]} -$ with $q\mapsto 1$ to it (this process is called the \emph{specialization at $q=1$}) \cite[Proposition 3.4 and Theorem 3.17]{KL1} (see also the diagram written before \cite[Lemma 5.3]{KOP2}). The free $\mathbb{Z}[q^{\pm 1}]$-module $\bigoplus_{{\bf d} \in \mathbb{Z}^I _{\ge 0}}G_0(R_{\bf d}\text{-}{\rm gmod})$ has the basis consisting of the classes of self-dual graded simple modules. Here we call this basis the {\it KLR-basis}. Then the specialization of the KLR-basis at $q=1$ is known to be a perfect basis \cite[Lemmas 3.13 and 5.3]{KOP2} (cf.~\cite[\S\S 2.5.1]{LV}). The property (v) holds because the involution $\ast$ is induced from the twist of $R_{\bf d}$-modules by the involutive automorphism $\sigma$ of $R_{\bf d}$ in \cite[\S\S 2.1]{KL1} (see also \cite[Section 12]{McN}). 
\end{ex}\vspace{2mm}

The index set $\mathcal{B}(\lambda)$ (resp., $\mathcal{B}(\infty)$) of the upper global basis in Example \ref{e:upper global} has the following additional structure, called a \emph{crystal} structure \cite[Section 1]{Kas4}: maps $\varepsilon_i, \varphi_i\colon \mathcal{B}(\lambda) \to \mathbb{Z}$ (resp., $\mathcal{B}(\infty) \to \mathbb{Z}$) and $\tilde{e}_i, \tilde{f}_i \colon \mathcal{B}(\lambda) \to \mathcal{B}(\lambda) \cup \{0\}$ (resp., $\mathcal{B}(\infty) \to \mathcal{B}(\infty) \cup \{0\}$) for $i \in I$. We do not review here the precise definitions of these crystals (see \cite{Kas5} for a survey on this topic). Instead, we explain the definition of the crystals associated with perfect bases. In fact, it is known that they are isomorphic to $\mathcal{B}(\lambda)$ and $\mathcal{B}(\infty)$ as crystals (Proposition \ref{p:uniqueness}). Let ${\bf B}^{\rm up}(\lambda)$ (resp., ${\bf B}^{\rm up}$) be a perfect basis of $V(\lambda)^\ast$, $\lambda\in P_+$ (resp., $\mathbb{C}[U^-]$). For $i\in I$ and $\tau \in {\bf B}^{\rm up}(\lambda)_{\mu}$ (resp., $\tau \in {\bf B}^{\rm up}_{\bf d}$), set 
\begin{align*}
\wt(\tau)&\coloneqq\mu\ (\text{resp.,}\ -\sum\nolimits_{i\in I}d_i\alpha_i),  
&
\varphi_i(\tau)&\coloneqq \varepsilon_i(\tau)+\langle \wt (\tau), h_i\rangle,&
\tilde{f}_i(\tau)&\coloneqq
\begin{cases}
\tau'&\text{if}\ \tilde{e}_i(\tau')=\tau,\\
0&\text{otherwise}. 
\end{cases}
\end{align*}
Then the sextuple $({\bf B}^{\rm up}(\lambda); \wt, \{\varepsilon_i\}_{i}, \{\varphi_i\}_{i}, \{\tilde{e}_i\}_{i}, \{\tilde{f}_i\}_{i})$ (resp.,~$({\bf B}^{\rm up}; \wt, \{\varepsilon_i\}_{i}, \{\varphi_i\}_{i}, \{\tilde{e}_i\}_{i}, \{\tilde{f}_i\}_{i})$) satisfies the axioms of crystal. 

\vspace{2mm}\begin{prop}[{\cite[Main Theorem 5.37]{BK}}]\label{p:uniqueness}
The following hold:
\begin{enumerate}
\item[{\rm (1)}] For $\lambda\in P_+$, the crystal $({\bf B}^{\rm up}(\lambda); \wt, \{\varepsilon_i\}_{i}, \{\varphi_i\}_i, \{\tilde{e}_i\}_i, \{\tilde{f}_i\}_i)$ is canonically isomorphic to the crystal $(\mathcal{B}(\lambda); \wt, \{\varepsilon_i\}_i, \{\varphi_i\}_i, \{\tilde{e}_i\}_i, \{\tilde{f}_i\}_i)$, that is, there exists a unique bijection ${\bf B}^{\rm up}(\lambda)\xrightarrow{\sim} \mathcal{B}(\lambda)$ that commutes with the maps $\{\tilde{e}_i\mid i\in I\}$, $\{\tilde{f}_i\mid i\in I\}$, and preserves the values of $\wt$, $\{\varepsilon_i \mid i \in I\}$, $\{\varphi_i\mid i\in I\}$. 
\item[{\rm (2)}] The crystal $({\bf B}^{\rm up}; \wt, \{\varepsilon_i\}_i, \{\varphi_i\}_i, \{\tilde{e}_i\}_i, \{\tilde{f}_i\}_i)$ is canonically isomorphic to the crystal $(\mathcal{B}(\infty); \wt,$ $\{\varepsilon_i\}_i, \{\varphi_i\}_i, \{\tilde{e}_i\}_i, \{\tilde{f}_i\}_i)$. 
\end{enumerate}
\end{prop}\vspace{2mm}

Remark that
\begin{align*}
&\varepsilon_i(b) = \max\{k \in \mathbb{Z}_{\ge 0} \mid \tilde{e}_i ^k b \neq 0\}\ \text{for}\ b \in \mathcal{B}(\infty),\ \text{and}\\
&\varepsilon_i(b) = \max\{k \in \mathbb{Z}_{\ge 0} \mid \tilde{e}_i ^k b \neq 0\},\ \varphi_i (b) = \max\{k \in \mathbb{Z}_{\ge 0} \mid \tilde{f}_i ^k b \neq 0\}\ \text{for}\ b \in \mathcal{B}(\lambda).
\end{align*}
From now on, by Proposition \ref{p:uniqueness}, we write perfect bases of $V(\lambda)^{\ast}$ and $\mathbb{C}[U^-]$ as $\{\Xi_{\lambda}^{\rm up}(b) \mid b \in \mathcal{B}(\lambda)\}$ and $\{\Xi^{\rm up}(b) \mid b \in \mathcal{B}(\infty)\}$. The unique element $b$ of $\mathcal{B}(\lambda)$ (resp., $\mathcal{B}(\infty)$) with $\wt(b)=\lambda$ (resp., $\wt(b)=0$) is denoted by $b_\lambda$ (resp., $b_\infty$). 

\vspace{2mm}\begin{rem}\normalfont\label{r:length one case}
Let ${\bf e}_i \in \mathbb{Z}^I _{\ge 0}$ denote the unit vector corresponding to $i \in I$. Since $U(\mathfrak{u}^-)_{k{\bf e}_i} = \mathbb{C} f_i ^k$ for $k \ge 0$, we have $U(\mathfrak{u}^-)^\ast _{{\rm gr}, k{\bf e}_i} = \mathbb{C} \Xi^{\rm up}(\tilde{f}_i ^k b_\infty)$.
\end{rem}\vspace{2mm}

Now the condition (iii) in Definition \ref{definition of perfect bases} is equivalent to the following condition:
\begin{enumerate}
\item[{\rm (iii)$'$}] for all $i \in I$, $b \in \mathcal{B}(\infty)$, and $k \in \mathbb{Z}_{\ge 0}$, \[f_i ^k \cdot \Xi^{\rm up}(b) \in \mathbb{C}^\times \Xi^{\rm up}(\tilde{e}_i ^k b) + \sum_{\substack{b^\prime \in \mathcal{B}(\infty);\ {\rm wt}(b^\prime) = {\rm wt}(\tilde{e}_i ^k b),\\ \varepsilon_i (b^\prime) < \varepsilon_i (\tilde{e}_i ^k b)}} \mathbb{C} \Xi^{\rm up}(b^\prime).\]
\end{enumerate} 
Moreover we have
\begin{align}
&f_i ^{\varepsilon_i(b)} \cdot \Xi^{\rm up}(b) \in \mathbb{C}^\times \Xi^{\rm up}(\tilde{e}_i ^{\varepsilon_i(b)} b),\label{string1}\\ 
&f_i ^k \cdot \Xi^{\rm up}(b)=0\;\text{for }k>\varepsilon_i(b).\label{string2}
\end{align}
A perfect basis ${\bf B}^{\rm up}(\lambda)$ also has the similar properties, but we do not need them in this paper. The following lemma follows by using \eqref{string1} and \eqref{string2} repeatedly.
\vspace{2mm}\begin{lem}\label{l:triangle}
Let ${\bf B}_k^{\rm up} = \{\Xi_k^{\rm up}(b) \mid b \in \mathcal{B}(\infty)\}$ be perfect bases of $\mathbb{C}[U^-]$ $(k=1, 2)$. Write 
\[
\Xi_1^{\rm up}(b)=\sum_{b^{\prime}\in\mathcal{B}(\infty)}c_{b, b^{\prime}}\Xi_2^{\rm up}(b')\quad(c_{b, b^{\prime}}\in\mathbb{C}).
\]
Then $c_{b, b}\in \mathbb{C}^{\times}$. Moreover $c_{b, b^{\prime}}=0$ unless $\Phi_{\bf i}(b^{\prime})\leq \Phi_{\bf i}(b)$ $($see Definition \ref{d:order} for the definition of the order $<)$, where $\Phi_{\bf i}\colon \mathcal{B}(\infty) \rightarrow \mathbb{Z}_{\ge 0} ^N$ is the Littelmann string parametrization map associated with a reduced word ${\bf i} \in I^N$ for the longest element $w_0$ of $W$ $($see Definition \ref{d:parametrization} (1)$)$. 
\end{lem}\vspace{2mm}

\begin{rem}\normalfont
This lemma holds for perfect bases which do not necessarily have the property {\rm (v)}.
\end{rem}\vspace{2mm}

In the following, the dual basis of a perfect basis ${\bf B}^{\rm up}(\lambda)$ (resp., ${\bf B}^{\rm up}$) is also an important object, which is called a \emph{lower perfect basis}, and denoted by ${\bf B}^{\rm low}(\lambda) = \{\Xi_{\lambda}^{\rm low}(b) \mid b \in \mathcal{B}(\lambda)\} \subset V(\lambda)$ (resp., ${\bf B}^{\rm low}=\{\Xi^{\rm low}(b) \mid b \in \mathcal{B}(\infty)\}\subset U(\mathfrak{u}^-)$). Then the condition (iii)$'$ above is replaced by the following condition (see the proof of \cite[Lemma 4.6]{FN}): 
\begin{enumerate}
\item[{\rm $\text{(iii)}'_l$}] for all $i \in I$, $b \in \mathcal{B}(\infty)$, and $k \in \mathbb{Z}_{\ge 0}$, \[f_i ^k \cdot \Xi^{\rm low}(b) \in \mathbb{C}^\times \Xi^{\rm low}(\tilde{f}_i ^k b) + \sum_{\substack{b^\prime \in \mathcal{B}(\infty);\ {\rm wt}(b^\prime) = {\rm wt}(\tilde{f}_i ^k b),\\ \varepsilon_i (b^\prime) > \varepsilon_i (\tilde{f}_i ^k b)}} \mathbb{C} \Xi^{\rm low}(b^\prime).\] 
\end{enumerate}

\vspace{2mm}\begin{rem}\normalfont
Baumann introduced the notion of {\it bases of canonical type} in \cite{Bau}. The axioms of bases of canonical type are slightly stronger than our conditions on the lower perfect bases because they impose an additional condition on the coefficient of $\Xi^{\rm low}(\tilde{f}_i ^k b)$ in our condition $\text{(iii)}'_l$.
\end{rem}\vspace{2mm}

The dual bases of $\{G_{\lambda}^{\rm up}(b) \mid b \in \mathcal{B}(\lambda)\}$ and $\{G^{\rm up}(b) \mid b \in \mathcal{B}(\infty)\}$ (that is, the lower global bases of $V(\lambda)$ and $U(\mathfrak{u}^-)$) are denoted by $\{G_{\lambda}^{\rm low}(b) \mid b \in \mathcal{B}(\lambda)\}$ and $\{G^{\rm low}(b) \mid b \in \mathcal{B}(\infty)\}$, respectively. 

Recall the involution $\ast\colon U(\mathfrak{u}^-)\to U(\mathfrak{u}^-)$. We see from \cite[Theorem 2.1.1]{Kas4} that, for $b\in \mathcal{B}(\infty)$, there exists an element $b^{\ast}\in \mathcal{B}(\infty)$ such that $G^{\rm low}(b)^{\ast}=G^{\rm low}(b^{\ast})$ (see, for instance, \cite[Proposition 2.8 (1)]{FN}). The involution $\ast \colon \mathcal{B}(\infty) \to \mathcal{B}(\infty)$ is called {\it Kashiwara's involution}. Set \[\tilde{e}_i ^\ast \coloneqq \ast \circ \tilde{e}_i \circ \ast,\;\tilde{f}_i ^\ast \coloneqq \ast \circ \tilde{f}_i \circ \ast,\;\varepsilon_i^{\ast}\coloneqq \varepsilon_i\circ \ast,\ {\rm and}\ \varphi_i^{\ast}\coloneqq \varphi_i\circ\ast\] for $i \in I$. Note that $G^{\rm up}(b)^{\ast}=G^{\rm up}(b^{\ast})$ for $b\in \mathcal{B}(\infty)$. In fact all perfect bases have such a property as follows.

\vspace{2mm}\begin{prop}\label{p:*_property}
Let $\{\Xi^{\rm up}(b) \mid b \in \mathcal{B}(\infty)\}$ be a perfect basis of $\mathbb{C}[U^-]$. Then $\Xi^{\rm up}(b)^\ast =\Xi^{\rm up}(b^\ast)$ for all $b\in \mathcal{B}(\infty)$; hence the equality $\Xi^{\rm low}(b)^\ast =\Xi^{\rm low}(b^\ast)$ also holds for all $b\in \mathcal{B}(\infty)$.
\end{prop}

\begin{proof}
For $b \in \mathcal{B}(\infty)$, there exists $b^{\star} \in \mathcal{B}(\infty)$ such that $\Xi^{\rm up}(b)^{\ast}=\Xi^{\rm up}(b^{\star})$ by the property (v). Suppose that there exists $b\in \mathcal{B}(\infty)$ such that $(b^{\ast})^{\star}\neq b$. Let ${\bf i}$ be a reduced word for the longest element $w_0$, and $b_0$ an element such that $(b_0^{\ast})^{\star}\neq b_0$ and such that $\Phi_{\bf i}(b_0)\ge \Phi_{\bf i}(b)$ for all $b\in \mathcal{B}(\infty)_{\wt(b_0)}$ with $(b^{\ast})^{\star}\neq b$. Then 
\[
\langle \Xi^{\rm up}((b_0^{\ast})^{\star}), G^{\rm low}(b_0)\rangle=\langle\Xi^{\rm up}(b_0^{\ast})^\ast, G^{\rm low}(b_0)\rangle=\langle\Xi^{\rm up}(b_0^{\ast}), G^{\rm low}(b_0^{\ast})\rangle\neq 0
\]
by Lemma \ref{l:triangle}. Hence, by Lemma \ref{l:triangle} again, $\Phi_{\bf i}(b_0)<\Phi_{\bf i}((b_0^{\ast})^{\star})$. Hence, by the assumption on $b_0$, we have the equality $(b_0^{\ast})^{\star}=(((b_0^{\ast})^{\star})^{\ast})^{\star}$, which is equivalent to $b_0=(b_0^{\ast})^{\star}$. This contradicts the choice of $b_0$. Hence $(b^{\ast})^{\star}=b$ for all $b\in \mathcal{B}(\infty)$.
\end{proof}

By condition $\text{(iii)}'_l$ and Proposition \ref{p:*_property}, we obtain the following (see, for instance, the proof of \cite[Proposition 2.8]{FN}).

\vspace{2mm}\begin{prop}\label{p:rightaction}
For all $i \in I$, $b \in \mathcal{B}(\infty)$, and $k \in \mathbb{Z}_{\ge 0}$, \[\Xi^{\rm low} (b) \cdot f_i ^k \in \mathbb{C}^\times \Xi^{\rm low} ((\tilde{f}_i ^\ast)^k b) + \sum_{\substack{b^\prime \in \mathcal{B}(\infty);\ {\rm wt}(b^\prime) = {\rm wt}((\tilde{f}_i ^\ast)^k b),\\ \varepsilon_i ^\ast (b^\prime) > \varepsilon_i ^\ast ((\tilde{f}_i ^\ast)^k b)}} \mathbb{C} \Xi^{\rm low} (b^\prime).\]
\end{prop}\vspace{2mm}

We will prove that a perfect basis ${\bf B}^{\rm up}$ of $\mathbb{C}[U^-]$ induces a perfect basis ${\bf B}^{\rm up}(\lambda)$ of $V(\lambda)^\ast$. To do this, we here recall the remarkable properties of the lower global bases. 

\vspace{2mm}\begin{prop}[{\cite[Theorem 5]{Kas2}}]\label{p:crystal basis}
For $\lambda \in P_+$, let $\pi_\lambda \colon U(\mathfrak{u}^-) \twoheadrightarrow V(\lambda)$ denote the surjective $U(\mathfrak{u}^-)$-module homomorphism given by $u \mapsto u v_{\lambda}$.
\begin{enumerate}
\item[{\rm (1)}] For $b\in \mathcal{B}(\infty)$, there exists $\pi_{\lambda}(b) \in \mathcal{B}(\lambda) \cup \{0\}$ such that $\pi_{\lambda}(G^{\rm low}(b))=G_{\lambda}^{\rm low}(\pi_{\lambda}(b))$, where $G_{\lambda}^{\rm low}(0) \coloneqq 0$; in addition, for \[\widetilde{\mathcal{B}}(\lambda) \coloneqq \{b \in \mathcal{B}(\infty) \mid \pi_\lambda(b) \neq 0\},\] the map $\pi_\lambda \colon \widetilde{\mathcal{B}}(\lambda) \rightarrow \mathcal{B}(\lambda)$, $b \mapsto \pi_\lambda(b)$, is bijective.
\item[{\rm (2)}] $\tilde{f}_i \pi_\lambda(b) = \pi_\lambda (\tilde{f}_i b)$ for all $i \in I$ and $b \in \mathcal{B}(\infty)$.
\item[{\rm (3)}] $\tilde{e}_i \pi_\lambda(b) = \pi_\lambda (\tilde{e}_i b)$ for all $i \in I$ and $b \in \widetilde{\mathcal{B}}(\lambda)$. 
\item[{\rm (4)}] $\varepsilon_i (\pi_\lambda(b)) = \varepsilon_i (b)$ and $\varphi_i (\pi_\lambda (b)) = \varphi_i (b) + \langle\lambda, h_i\rangle$ for all $i \in I$ and $b \in \widetilde{\mathcal{B}}(\lambda)$. 
\end{enumerate}
\end{prop}

\vspace{2mm}\begin{prop}[{\cite[Theorem 7]{Kas2}}]\label{p:f_decomposition}
Let $k\in \mathbb{Z}_{\ge 0}$, and $i\in I$. Then 
\begin{align*}
\sum_{k'\ge k}f_i^{k'}U(\mathfrak{u}^-)&=\bigoplus_{b\in \mathcal{B}(\infty);\ \varepsilon_i(b)\ge k}\mathbb{C}G^{\rm low}(b),\;\text{and}\\
\sum_{k'\ge k}U(\mathfrak{u}^-)f_i^{k'}&=\bigoplus_{b\in \mathcal{B}(\infty);\ \varepsilon_i^{\ast}(b)\ge k}\mathbb{C}G^{\rm low}(b).
\end{align*}
\end{prop}\vspace{2mm}

\begin{rem}\label{r:Btilde}\normalfont
It is known that the kernel of the map $\pi_\lambda \colon U(\mathfrak{u}^-) \twoheadrightarrow V(\lambda)$ is equal to $\sum_{i\in I}U(\mathfrak{u}^-)f_i^{\langle\lambda, h_i\rangle+1}$. Hence, by Proposition \ref{p:f_decomposition}, the subset $\widetilde{\mathcal{B}}(\lambda)$ is described in terms of the crystal, that is,
\[
\widetilde{\mathcal{B}}(\lambda)=\{b\in \mathcal{B}(\infty)\mid \varepsilon_i ^\ast(b)\leq \langle\lambda, h_i\rangle\;\text{for all}\;i\in I\}.
\]
\end{rem}\vspace{2mm}

A lower perfect basis is compatible with irreducible highest weight $U(\mathfrak{g})$-modules $V(\lambda)$, $\lambda \in P_+$, and their $U(\mathfrak{g}_i)$-submodules as follows.

\vspace{2mm}\begin{prop}\label{p:compatibility}
Let $\lambda\in P_+$, and $w\in W$. Then the following hold.
\begin{enumerate}
\item[{\rm (1)}] For $b\in \mathcal{B}(\infty)$, we have $\pi_\lambda(\Xi^{\rm low}(b))\neq 0$ if and only if $b \in \widetilde{\mathcal{B}}(\lambda)$. Thus the set $\{\Xi_{\lambda}^{\rm low}(\pi_\lambda (b)) \coloneqq \pi_\lambda(\Xi^{\rm low}(b)) \mid b \in \widetilde{\mathcal{B}}(\lambda)\}$ forms a $\mathbb{C}$-basis of $V(\lambda)$.
\item[{\rm (2)}] Fix $i\in I$ and $\ell\in \mathbb{Z}_{\ge 0}$. Let $I_i^{\ell}(V(\lambda))$ be the sum of $(\ell+1)$-dimensional irreducible $U(\mathfrak{g}_i)$-submodules of $V(\lambda)$, $W_i^{\ell}(V(\lambda)) \coloneqq \bigoplus_{\ell'\ge \ell}I_i^{\ell'}(V(\lambda))$, $I_i^{\ell}(\mathcal{B}(\lambda)) \coloneqq \{b\in \mathcal{B}(\lambda)\mid \varepsilon_i(b)+\varphi_i(b)=\ell\}$, and $W_i^{\ell}(\mathcal{B}(\lambda)) \coloneqq \{b\in \mathcal{B}(\lambda)\mid \varepsilon_i(b)+\varphi_i(b)\ge\ell\}$. Then 
\[
W_i^{\ell}(V(\lambda))=\sum_{b\in W_i^{\ell}(\mathcal{B}(\lambda))}\mathbb{C}\Xi_{\lambda}^{\rm low}(b),
\]
and, for $b\in I_i^{\ell}(\mathcal{B}(\lambda))$,
\begin{align*}
f_i ^k \cdot \Xi_{\lambda}^{\rm low}(b) &\in \mathbb{C}^\times \Xi_{\lambda}^{\rm low}(\tilde{f}_i ^k b) + W_i^{\ell+1}(V(\lambda)),\\
e_i ^k \cdot \Xi_{\lambda}^{\rm low}(b) &\in \mathbb{C}^\times \Xi_{\lambda}^{\rm low}(\tilde{e}_i ^k b) + W_i^{\ell+1}(V(\lambda)).
\end{align*}
\end{enumerate}
\end{prop}
\begin{proof}
We first prove the assertion that the set $\{\Xi^{\rm low}(b) \mid b\in \mathcal{B}(\infty),\ \varepsilon_i ^{\ast}(b) \ge k\}$ forms a $\mathbb{C}$-basis of $U(\mathfrak{u}^-)f_i ^k$, which implies part (1) by  Remark \ref{r:Btilde}. Set $\tilde{U}_{i, k} \coloneqq \sum_{b \in \mathcal{B}(\infty);\ \varepsilon_i^{\ast}(b)\ge k}\mathbb{C}\Xi^{\rm low}(b)$. Then we have $U(\mathfrak{u}^-) f_i ^k \subset \tilde{U}_{i, k}$, by Proposition \ref{p:rightaction}. On the other hand, Proposition \ref{p:f_decomposition} implies that $\dim_\mathbb{C} (U(\mathfrak{u}^-) f_i ^k \cap U(\mathfrak{u}^-)_{\bf d})=\# \{b\in \mathcal{B}(\infty)_{\bf d} \mid \varepsilon_i^{\ast}(b)\ge k\}$ for all ${\bf d} = (d_i)_{i \in I} \in \mathbb{Z}^I _{\ge 0}$, where $\mathcal{B}(\infty)_{\bf d} \coloneqq \{b \in \mathcal{B}(\infty) \mid \wt(b) = -\sum_{i \in I} d_i \alpha_i\}$. This completes a proof of our assertion. By Proposition \ref{p:crystal basis} and the condition $\text{(iii)}'_l$ for lower perfect bases, we have
\begin{align}
f_i ^k \cdot \Xi_{\lambda}^{\rm low}(b) \in \mathbb{C}^\times \Xi_{\lambda}^{\rm low}(\tilde{f}_i ^k b) + \sum_{\substack{b^\prime \in \mathcal{B}(\lambda);\ {\rm wt}(b^\prime) = {\rm wt}(\tilde{f}_i ^k b),\\ \varepsilon_i (b^\prime) > \varepsilon_i (\tilde{f}_i ^k b)}} \mathbb{C} \Xi_{\lambda}^{\rm low}(b^\prime)\label{f_action}
\end{align}
for all $i\in I$, $b\in \mathcal{B}(\lambda)$, and $k\in \mathbb{Z}_{\ge 0}$. 
Fix $i\in I$ and let $\ell_0$ be the maximal integer $\ell$ such that $W_i^{\ell}(V(\lambda))\neq 0$. Then 
\begin{align*}
W_i^{\ell_0}(V(\lambda))=I_i^{\ell_0}(V(\lambda)).\label{sl2string}
\end{align*}
Thus $\varepsilon_i(b) = 0$ for all $b \in B(\lambda)$ with $\langle\wt (b), h_i\rangle=\ell_0$. Hence \eqref{f_action} implies that $f_i ^k \cdot \Xi_{\lambda}^{\rm low}(b) \in \mathbb{C}^\times \Xi_{\lambda}^{\rm low}(\tilde{f}_i ^k b)$ for all $k\in \mathbb{Z}_{\ge 0}$ and $b\in B(\lambda)$ with $\langle\wt (b), h_i\rangle=\ell_0$. Therefore $W_i^{\ell_0}(V(\lambda))$ is spanned by the elements $\{\Xi_{\lambda}^{\rm low}(b)\mid b\in W_i^{\ell_0}(\mathcal{B}(\lambda))\}$. By using descending induction on $\ell$ and replacing $V(\lambda)$ with $V(\lambda)/W_i^{\ell+1}(V(\lambda))$ in the argument above, we prove that $W_i^{\ell}(V(\lambda))$ is spanned by the elements $\{\Xi_{\lambda}^{\rm low}(b)\mid b\in W_i^{\ell}(\mathcal{B}(\lambda))\}$ for all $\ell$. This proves the first half of part {\rm (2)}. The latter half of part {\rm (2)} follows from \eqref{f_action}, the first half of {\rm (2)}, and the representation theory of $\mathfrak{sl}_2(\mathbb{C})$.
\end{proof}

From now on,  we review the main results of \cite{FN} and \cite{Kav}.

\vspace{2mm}\begin{prop}[{See \cite[Propositions 3.2.3 and 3.2.5]{Kas4}}]\label{Demazure crystal}
Let ${\bf i} = (i_1, \ldots, i_r)$ be a reduced word for $w \in W$, and $\lambda \in P_+$.
\begin{enumerate}
\item[{\rm (1)}] The subset \[\mathcal{B}_w(\lambda) \coloneqq \{\tilde{f}_{i_1} ^{a_1} \cdots \tilde{f}_{i_r} ^{a_r} b_\lambda \mid a_1, \ldots, a_r \in \mathbb{Z}_{\ge 0}\} \setminus \{0\} \subset \mathcal{B}(\lambda)\] is independent of the choice of a reduced word ${\bf i}$.
\item[{\rm (2)}] The subset \[\mathcal{B}_w(\infty) \coloneqq \{\tilde{f}_{i_1} ^{a_1} \cdots \tilde{f}_{i_r} ^{a_r} b_\infty \mid a_1, \ldots, a_r \in \mathbb{Z}_{\ge 0}\} \subset \mathcal{B}(\infty)\] is independent of the choice of a reduced word ${\bf i}$.
\item[{\rm (3)}] The equality $\pi_\lambda(\mathcal{B}_w(\infty)) = \mathcal{B}_w(\lambda) \cup \{0\}$ holds; hence $\pi_\lambda$ induces a bijective map $\pi_\lambda \colon \widetilde{\mathcal{B}}_w(\lambda) \rightarrow \mathcal{B}_w(\lambda)$, where $\widetilde{\mathcal{B}}_w(\lambda) \coloneqq \mathcal{B}_w(\infty) \cap \widetilde{\mathcal{B}}(\lambda)$.
\end{enumerate}
\end{prop}\vspace{2mm}

The subsets $\mathcal{B}_w (\lambda), \mathcal{B}_w(\infty)$ are called {\it Demazure crystals}. 

Since 
\begin{align}\label{covering lambda}
\mathcal{B}(\infty) = \bigcup_{\lambda \in P_+} \widetilde{\mathcal{B}}(\lambda) 
\end{align}
by \cite[Corollary 4.4.5]{Kas2}, we deduce that 
\begin{equation}\label{covering lambda w}
\begin{aligned}
\mathcal{B}_w(\infty) &= \bigcup_{\lambda \in P_+} \mathcal{B}_w(\infty) \cap \widetilde{\mathcal{B}}(\lambda)\\
&= \bigcup_{\lambda \in P_+} \widetilde{\mathcal{B}}_w(\lambda).
\end{aligned}
\end{equation}
Let $\{\Xi^{\rm up} _\lambda (b) \mid b \in \mathcal{B}(\lambda)\} \subset H^0(G/B, \mathcal{L}_\lambda) = V(\lambda)^\ast$ denote the dual basis of $\{\Xi^{\rm low} _\lambda (b) \mid b \in \mathcal{B}(\lambda)\} \subset V(\lambda)$.

\vspace{2mm}\begin{prop}\label{dual canonical lambda}
Let $\lambda \in P_+$.
\begin{enumerate}
\item[{\rm (1)}] The $\mathbb{C}$-basis ${\bf B}^{\rm up}(\lambda) = \{\Xi^{\rm up} _\lambda (b) \mid b \in \mathcal{B}(\lambda)\} \subset V(\lambda)^\ast$ is a perfect basis.
\item[{\rm (2)}] Set $\tau_\lambda \coloneqq \Xi^{\rm up} _\lambda (b_\lambda) \in H^0(G/B, \mathcal{L}_\lambda)$. Then the section $\tau_\lambda$ does not vanish on $U^-\ (\hookrightarrow G/B)$; in particular, the restriction $(\tau/\tau_\lambda)|_{U^-}$ belongs to $\mathbb{C}[U^-]$ for all $\tau \in H^0(G/B, \mathcal{L}_\lambda)$. 
\item[{\rm (3)}] A map $\iota_\lambda \colon H^0(G/B, \mathcal{L}_\lambda) \rightarrow \mathbb{C}[U^-]$ defined by $\iota_\lambda(\tau) \coloneqq (\tau/\tau_\lambda)|_{U^-}$ for $\tau \in H^0(G/B, \mathcal{L}_\lambda)$ is injective.
\item[{\rm (4)}] The element $\Xi^{\rm up}(b)$ is identical to $\iota_\lambda(\Xi^{\rm up} _\lambda(\pi_\lambda(b)))$ for all $b \in \widetilde{\mathcal{B}}(\lambda)$.
\item[{\rm (5)}] The following equalities hold: 
\begin{align*}
&\{\Xi^{\rm up}(b) \mid b \in \mathcal{B}(\infty)\} = \bigcup_{\lambda \in P_+} \{\iota_\lambda(\Xi^{\rm up} _\lambda (b)) \mid b \in \mathcal{B}(\lambda)\},\ {\it and}\\
&\mathbb{C}[U^-] = \bigcup_{\lambda \in P_+} \iota_\lambda(H^0(G/B, \mathcal{L}_\lambda)).
\end{align*}
\end{enumerate}
\end{prop}

\begin{proof}
Part (1) is an immediate consequence of the definition of ${\bf B}^{\rm up}(\lambda)$ and \eqref{f_action}. Parts (2), (4) are proved in a way similar to the proof of \cite[Lemma 4.5]{FN}. Since $U^-$ is regarded as an open subvariety of $G/B$, we have $(\tau/\tau_\lambda)|_{U^-} \neq 0$ for all nonzero sections $\tau \in H^0(G/B, \mathcal{L}_\lambda)$, which implies part (3). Since $\{\Xi^{\rm up}(b) \mid b \in \mathcal{B}(\infty)\}$ is a $\mathbb{C}$-basis of $\mathbb{C}[U^-]$, part (5) follows by part (4) and equation (\ref{covering lambda}).
\end{proof}

We consider the following property (D) for a lower perfect basis $\{\Xi^{\rm low}(b) \mid b \in \mathcal{B}(\infty)\}$ (see also Proposition \ref{p:compatibility} (1)):
\begin{enumerate}
\item[(D)] the set $\{\Xi_{\lambda}^{\rm low}(b) \mid b \in \mathcal{B}_w(\lambda)\}$ forms a $\mathbb{C}$-basis of the Demazure module $V_w(\lambda)$.
\end{enumerate}
A perfect basis of $\mathbb{C}[U^-]$ will be said to \emph{have the property (D)} if its dual basis has the property (D). 

\vspace{2mm}\begin{rem}\label{r:D_example}\normalfont
The upper global basis and the dual semicanonical basis have the property (D) (\cite[Proposition 3.2.3]{Kas4} and \cite[Theorem 7.1]{Sav}, respectively). We show in Section 4 (Proposition \ref{p:pos_D}) that the specialization of the KLR-basis at $q=1$ also has the property (D). 
\end{rem}\vspace{2mm}

Since $U^- \cap X(w)$ is a closed subvariety of $U^-$, the restriction map $\eta_w \colon \mathbb{C}[U^-] \twoheadrightarrow \mathbb{C}[U^- \cap X(w)]$ is surjective. For $b \in \mathcal{B}(\infty)$, let $\Xi_w ^{\rm up} (b) \in \mathbb{C}[U^- \cap X(w)]$ denote the image of $\Xi^{\rm up} (b) \in \mathbb{C}[U^-]$ under the restriction map $\eta_w$. By abuse of notation, we denote by $\tau_\lambda \in H^0(X(w), \mathcal{L}_\lambda)$ the restriction of $\tau_\lambda \in H^0(G/B, \mathcal{L}_\lambda)$. By Proposition \ref{dual canonical lambda} (2), the section $\tau_\lambda$ does not vanish on $U^- \cap X(w)\ (\hookrightarrow X(w))$; hence a map $H^0(X(w), \mathcal{L}_\lambda) \rightarrow \mathbb{C}[U^- \cap X(w)]$, $\tau \mapsto (\tau/\tau_\lambda)|_{(U^- \cap X(w))}$, is well-defined, which we also denote by $\iota_\lambda$. Since $U^- \cap X(w)$ is an open subvariety of $X(w)$, we see that the map $\iota_\lambda \colon H^0(X(w), \mathcal{L}_\lambda) \rightarrow \mathbb{C}[U^- \cap X(w)]$ is injective. For a perfect basis $\{\Xi^{\rm up}(b) \mid b \in \mathcal{B}(\infty)\}$ with the property (D), let $\{\Xi^{\rm up} _{\lambda, w}(b) \mid b \in \mathcal{B}_w(\lambda)\} \subset H^0(X(w), \mathcal{L}_\lambda) = V_w(\lambda)^\ast$ be the dual basis of $\{\Xi^{\rm low} _\lambda(b) \mid b \in \mathcal{B}_w(\lambda)\} \subset V_w(\lambda)$. It is obvious that $\tau_\lambda = \Xi^{\rm up} _{\lambda, w}(b_\lambda)$ in $H^0(X(w), \mathcal{L}_\lambda)$.

\vspace{2mm}\begin{cor}\label{vanishing}
Let ${\bf B}^{\rm up} = \{\Xi^{\rm up}(b) \mid b \in \mathcal{B}(\infty)\} \subset \mathbb{C}[U^-]$ be a perfect basis with the property (D).
\begin{enumerate}
\item[{\rm (1)}] The following equality holds: \[\mathbb{C}[U^- \cap X(w)] = \bigcup_{\lambda \in P_+} \iota_\lambda(H^0(X(w), \mathcal{L}_\lambda)).\] 
\item[{\rm (2)}] The element $\Xi^{\rm up} _w(b)$ is identical to $\iota_\lambda(\Xi^{\rm up} _{\lambda, w}(\pi_\lambda(b)))$ for all $b \in \widetilde{\mathcal{B}}_w(\lambda)$. 
\item[{\rm (3)}] The set $\{\Xi^{\rm up} _w(b) \mid b\in\mathcal{B}_w(\infty)\}$ forms a $\mathbb{C}$-basis of $\mathbb{C}[U^- \cap X(w)]$.
\item[{\rm (4)}] The element $\Xi^{\rm up} _w(b)$ is identical to $0$ unless $b \in \mathcal{B}_w (\infty)$.
\end{enumerate}
\end{cor}

\begin{proof}
Consider the following diagram of subvarieties: 
\begin{align*}
\xymatrix{U^- \ar@{^{(}->}[r] & G/B \\
U^- \cap X(w) \ar@{^{(}->}[u] \ar@{^{(}->}[r] & X(w). \ar@{^{(}->}[u]}
\end{align*}
From this, we see that the following diagram is commutative: 
\begin{align*}
\xymatrix{\mathbb{C}[U^-] \ar@{->>}[d]^-{\eta_w} & H^0(G/B, \mathcal{L}_{\lambda}) \ar@{_{(}->}[l]^-{\iota_\lambda} \ar@{->>}[d]^-{\eta_w}\\
\mathbb{C}[U^- \cap X(w)] & H^0(X(w), \mathcal{L}_{\lambda})\ar@{_{(}->}[l]^-{\iota_\lambda},}
\end{align*}
where we denote by $\eta_w \colon H^0(G/B, \mathcal{L}_\lambda) \twoheadrightarrow H^0(X(w), \mathcal{L}_\lambda)$ the restriction map. Hence part (2) is an immediate consequence of Proposition \ref{dual canonical lambda} (4) and of the equality $\eta_w (\Xi^{\rm up} _{\lambda}(\pi_\lambda(b))) = \Xi^{\rm up} _{\lambda, w}(\pi_\lambda(b))$ for $b \in \widetilde{\mathcal{B}}_w(\lambda)$. Also, we see that 
\begin{align*}
\mathbb{C}[U^- \cap X(w)] &= \eta_w(\mathbb{C}[U^-])\\
&= \bigcup_{\lambda \in P_+} \eta_w(\iota_\lambda(H^0(G/B, \mathcal{L}_\lambda)))\quad({\rm by\ Proposition}\ \ref{dual canonical lambda}\ (5))\\
&= \bigcup_{\lambda \in P_+} \iota_\lambda(\eta_w(H^0(G/B, \mathcal{L}_\lambda)))\\
&= \bigcup_{\lambda \in P_+} \iota_\lambda(H^0(X(w), \mathcal{L}_\lambda)).
\end{align*}
This proves part (1). Since $\{\Xi^{\rm up} _{\lambda, w}(\pi_\lambda(b)) \mid b \in \widetilde{\mathcal{B}}_w(\lambda)\}$ forms a $\mathbb{C}$-basis of $H^0(X(w), \mathcal{L}_{\lambda})$, we deduce by parts (1), (2) and equation (\ref{covering lambda w}) that $\{\Xi^{\rm up} _w(b) \mid \mathcal{B}_w(\infty)\}$ spans $\mathbb{C}[U^- \cap X(w)]$. For an arbitrary finite subset $\{b_1, \ldots, b_k\} \subset \mathcal{B}_w(\infty)$, take $\lambda \in P_+$ such that $b_1, \ldots, b_k \in \widetilde{B}(\lambda)$. Since $\{\Xi^{\rm up} _{\lambda, w}(\pi_\lambda(b_1)), \ldots, \Xi^{\rm up} _{\lambda, w}(\pi_\lambda(b_k))\}$ is linearly independent, it follows by part (2) that $\{\Xi^{\rm up} _w(b_1), \ldots, \Xi^{\rm up} _w(b_k)\}$ is also linearly independent. From these, we obtain part (3). For $b \in \mathcal{B}(\infty) \setminus \mathcal{B}_w (\infty)$, we take $\lambda \in P_+$ such that $b \in \widetilde{\mathcal{B}}(\lambda)$. Since $\pi_\lambda \colon \widetilde{\mathcal{B}}(\lambda) \xrightarrow{\sim} \mathcal{B}(\lambda)$ is bijective and $\pi_\lambda(\widetilde{\mathcal{B}}_w(\lambda)) = \mathcal{B}_w(\lambda)$ by Proposition \ref{Demazure crystal} (3), we have $\pi_\lambda(b) \notin \mathcal{B}_w(\lambda)$, which implies that $\eta_w(\Xi^{\rm up} _{\lambda}(\pi_\lambda(b))) = 0$ by (D). Hence it holds that 
\begin{align*}
\Xi^{\rm up} _w(b) &= \eta_w(\Xi^{\rm up}(b))\\
&= \eta_w(\iota_\lambda(\Xi^{\rm up} _{\lambda}(\pi_\lambda(b))))\quad({\rm by\ Proposition}\ \ref{dual canonical lambda}\ (4))\\
&= \iota_\lambda(\eta_w(\Xi^{\rm up} _{\lambda}(\pi_\lambda(b))))\\
&= \iota_\lambda(0) = 0,
\end{align*}
which implies part (4). This proves the corollary.
\end{proof}

\begin{rem}\normalfont
Some formulas with respect to the character of $\mathbb{C}[U^- \cap X(w)]$ are given by \cite[\S\S 12.1]{Kum}. By Corollary \ref{vanishing} (3), these formulas can be regarded as those with respect to the character of $\mathcal{B}_w(\infty)$ (see \cite[\S\S 4.7]{Jos}).
\end{rem}

\vspace{2mm}\begin{defi}\label{d:parametrization}\normalfont
Let ${\bf i} = (i_1, \ldots, i_r) \in I^r$ be a reduced word for $w \in W$. 
\begin{enumerate}
\item[{\rm (1)}] For $b \in \mathcal{B}_w(\infty)$, define $\Phi_{\bf i} (b) = (a_1, \ldots, a_r) \in \mathbb{Z}_{\ge 0} ^r$ by
\begin{align*}
&a_1 \coloneqq \max\{a \in \mathbb{Z}_{\ge 0} \mid \tilde{e}_{i_1} ^a b \neq 0\},\\
&a_2 \coloneqq \max\{a \in \mathbb{Z}_{\ge 0} \mid \tilde{e}_{i_2} ^a \tilde{e}_{i_1} ^{a_1} b \neq 0\},\\
&\ \vdots\\
&a_r \coloneqq \max\{a \in \mathbb{Z}_{\ge 0} \mid \tilde{e}_{i_r} ^a \tilde{e}_{i_{r-1}} ^{a_{r-1}} \cdots \tilde{e}_{i_1} ^{a_1} b \neq 0\}.
\end{align*}
The $\Phi_{\bf i}(b)$ is called the {\it Littelmann string parametrization} of $b$ with respect to ${\bf i}$ (see \cite[Section 1]{Lit}). The map $\Phi_{\bf i} \colon \mathcal{B}_w(\infty) \rightarrow \mathbb{Z}_{\ge 0} ^r$ is indeed an injection.
\item[{\rm (2)}] For $b \in \mathcal{B}_w(\infty)$, define $\Psi_{\bf i} (b) = (a_r ^\prime, \ldots, a_1 ^\prime) \in \mathbb{Z}_{\ge 0} ^r$ by
\begin{align*}
&a_r ^\prime \coloneqq \max\{a \in \mathbb{Z}_{\ge 0} \mid (\tilde{e}_{i_r} ^\ast)^a b \neq 0\},\\
&a_{r-1} ^\prime \coloneqq \max\{a \in \mathbb{Z}_{\ge 0} \mid (\tilde{e}_{i_{r-1}} ^\ast)^a (\tilde{e}_{i_r} ^\ast)^{a_r ^\prime} b \neq 0\},\\
&\ \vdots\\
&a_1 ^\prime \coloneqq \max\{a \in \mathbb{Z}_{\ge 0} \mid (\tilde{e}_{i_1} ^\ast)^a (\tilde{e}_{i_2} ^\ast)^{a_2 ^\prime} \cdots (\tilde{e}_{i_r} ^\ast)^{a_r ^\prime} b \neq 0\}.
\end{align*}
The map $\Psi_{\bf i}$ is called the {\it Kashiwara embedding} of $\mathcal{B}_w(\infty)$. The map $\Psi_{\bf i} \colon \mathcal{B}_w(\infty) \rightarrow \mathbb{Z}_{\ge 0} ^r$ is indeed an injection (see \cite[Sections 2 and 3]{Kas4}).
\end{enumerate}
\end{defi}

\vspace{2mm}\begin{rem}\normalfont\label{parametrization for modules}
Through the bijective map $\pi_\lambda \colon \widetilde{\mathcal{B}}_w(\lambda) \xrightarrow{\sim} \mathcal{B}_w(\lambda)$ in Proposition \ref{Demazure crystal} (3), the maps $\Phi_{\bf i}$ and $\Psi_{\bf i}$ induce the Littelmann string parametrization for $\mathcal{B}_w(\lambda)$ and the Kashiwara embedding of $\mathcal{B}_w(\lambda)$.
\end{rem}

\vspace{2mm}\begin{defi}\normalfont
Let ${\bf i} \in I^r$ be a reduced word for $w \in W$, and $\lambda \in P_+$. 
\begin{enumerate}
\item[{\rm (1)}] Define a subset $\mathcal{S}_{\bf i} ^{(\lambda, w)} \subset \mathbb{Z}_{>0} \times \mathbb{Z}^r$ by \[\mathcal{S}_{\bf i} ^{(\lambda, w)} \coloneqq \bigcup_{k>0} \{(k, \Phi_{\bf i}(b)) \mid b \in \widetilde{\mathcal{B}}_w (k\lambda)\},\] and denote by $\mathcal{C}_{\bf i} ^{(\lambda, w)} \subset \mathbb{R}_{\ge 0} \times \mathbb{R}^r$ the smallest real closed cone containing $\mathcal{S}_{\bf i} ^{(\lambda, w)}$. Let us define a subset $\Delta_{\bf i} ^{(\lambda, w)} \subset \mathbb{R}^r$ by \[\Delta_{\bf i} ^{(\lambda, w)} \coloneqq \{{\bf a} \in \mathbb{R}^r \mid (1, {\bf a}) \in \mathcal{C}_{\bf i} ^{(\lambda, w)}\}.\] This subset $\Delta_{\bf i} ^{(\lambda, w)}$ is called the {\it Littelmann string polytope} for $\mathcal{B}_w(\lambda)$ with respect to ${\bf i}$ (see \cite[Definition 3.5]{Kav} and \cite[Section 1]{Lit}).
\item[{\rm (2)}] Define a subset $\widetilde{\mathcal{S}}_{\bf i} ^{(\lambda, w)} \subset \mathbb{Z}_{>0} \times \mathbb{Z}^r$ by \[\widetilde{\mathcal{S}}_{\bf i} ^{(\lambda, w)} \coloneqq \bigcup_{k>0} \{(k, \Psi_{\bf i}(b)) \mid b \in \widetilde{\mathcal{B}}_w (k\lambda)\},\] and denote by $\widetilde{\mathcal{C}}_{\bf i} ^{(\lambda, w)} \subset \mathbb{R}_{\ge 0} \times \mathbb{R}^r$ the smallest real closed cone containing $\widetilde{\mathcal{S}}_{\bf i} ^{(\lambda, w)}$. Let us define a subset $\widetilde{\Delta}_{\bf i} ^{(\lambda, w)} \subset \mathbb{R}^r$ by \[\widetilde{\Delta}_{\bf i} ^{(\lambda, w)} \coloneqq \{{\bf a} \in \mathbb{R}^r \mid (1, {\bf a}) \in \widetilde{\mathcal{C}}_{\bf i} ^{(\lambda, w)}\}.\] This subset $\widetilde{\Delta}_{\bf i} ^{(\lambda, w)}$ is called the {\it Nakashima-Zelevinsky polyhedral realization of} $\mathcal{B}_w(\lambda)$ associated with ${\bf i}$ (see \cite[\S\S 2.3]{FN}, \cite[Sections 3 and 4]{Nak1}, \cite[\S\S 3.1]{Nak2}, and \cite[Section 3]{NZ}).
\end{enumerate}
\end{defi}\vspace{2mm}

\begin{prop}[{See \cite[\S\S 3.2 and Theorem 3.10]{BZ} and \cite[Section 1]{Lit}}]\label{string lattice points}
The following hold.
\begin{enumerate}
\item[{\rm (1)}] The real closed cone $\mathcal{C}_{\bf i} ^{(\lambda, w)}$ is a rational convex polyhedral cone, that is, there exists a finite number of rational points ${\bf a}_1, \ldots, {\bf a}_l \in \mathbb{Q}_{\ge 0} \times \mathbb{Q}^r$ such that $\mathcal{C}_{\bf i} ^{(\lambda, w)} = \mathbb{R}_{\ge 0}{\bf a}_1 + \cdots + \mathbb{R}_{\ge 0} {\bf a}_l$. Moreover the equality $\mathcal{S}_{\bf i} ^{(\lambda, w)} = \mathcal{C}_{\bf i} ^{(\lambda, w)} \cap (\mathbb{Z}_{>0} \times \mathbb{Z}^r)$ holds. 
\item[{\rm (2)}] The set $\Delta_{\bf i} ^{(\lambda, w)}$ is a rational convex polytope, and the equality $\Phi_{\bf i} (\widetilde{\mathcal{B}}_w(\lambda)) = \Delta_{\bf i} ^{(\lambda, w)} \cap \mathbb{Z}^r$ holds.
\end{enumerate}
\end{prop}\vspace{2mm}

\begin{prop}[{See \cite[Corollary 4.3]{FN}}]\label{polyhedral lattice points}
The following hold.
\begin{enumerate}
\item[{\rm (1)}] The real closed cone $\widetilde{\mathcal{C}}_{\bf i} ^{(\lambda, w)}$ is a rational convex polyhedral cone, and the equality $\widetilde{\mathcal{S}}_{\bf i} ^{(\lambda, w)} = \widetilde{\mathcal{C}}_{\bf i} ^{(\lambda, w)} \cap (\mathbb{Z}_{>0} \times \mathbb{Z}^r)$ holds. 
\item[{\rm (2)}] The set $\widetilde{\Delta}_{\bf i} ^{(\lambda, w)}$ is a rational convex polytope, and the equality $\Psi_{\bf i} (\widetilde{\mathcal{B}}_w(\lambda)) = \widetilde{\Delta}_{\bf i} ^{(\lambda, w)} \cap \mathbb{Z}^r$ holds.
\end{enumerate}
\end{prop}\vspace{2mm}

\begin{rem}\normalfont
By \cite[Theorem 3.10]{BZ} and \cite[Section 1]{Lit}, we obtain a system of explicit linear inequalities defining the string polytope $\Delta_{\bf i} ^{(\lambda, w)}$. In addition, under a certain positivity assumption on ${\bf i}$, a system of explicit linear inequalities defining the Nakashima-Zelevinsky polyhedral realization $\widetilde{\Delta}_{\bf i} ^{(\lambda, w)}$ is given by \cite[Theorem 4.1]{Nak1} and \cite[Proposition 3.1]{Nak2} (see also \cite[Corollary 5.3]{FN}).
\end{rem}\vspace{2mm}

Define a linear automorphism $\omega \colon \mathbb{R} \times \mathbb{R}^r \xrightarrow{\sim} \mathbb{R} \times \mathbb{R}^r$ by $\omega(k, {\bf a}) \coloneqq (k, -{\bf a})$. Recall that $\tau_\lambda = \Xi^{\rm up} _{\lambda, w} (b_\lambda) \in H^0(X(w), \mathcal{L}_\lambda)$. By \cite[Section 4]{Kav}, we obtain the following.

\vspace{2mm}\begin{prop}\label{string polytopes}
Let ${\bf B}^{\rm up} = \{\Xi^{\rm up}(b) \mid b \in \mathcal{B}(\infty)\} \subset \mathbb{C}[U^-]$ be a perfect basis, ${\bf i} \in I^r$ a reduced word for $w \in W$, and $\lambda \in P_+$.
\begin{enumerate}
\item[{\rm (1)}] The Littelmann string parametrization $\Phi_{\bf i} (b)$ is equal to $-v_{\bf i} ^{\rm high}(\Xi^{\rm up} _w (b))$ for all $b \in \mathcal{B}_w(\infty)$.
\item[{\rm (2)}] The equalities $\mathcal{S}_{\bf i} ^{(\lambda, w)} = \omega(S(X(w), \mathcal{L}_\lambda, v_{\bf i} ^{\rm high}, \tau_\lambda))$, $\mathcal{C}_{\bf i} ^{(\lambda, w)} = \omega(C(X(w), \mathcal{L}_\lambda, v_{\bf i} ^{\rm high}, \tau_\lambda))$, and $\Delta_{\bf i} ^{(\lambda, w)} = -\Delta(X(w), \mathcal{L}_\lambda, v_{\bf i} ^{\rm high}, \tau_\lambda)$ hold.
\end{enumerate}
\end{prop}\vspace{2mm}

By \cite[Section 4]{FN}, we obtain the following.

\vspace{2mm}\begin{prop}\label{polyhedral realizations}
Let ${\bf B}^{\rm up} = \{\Xi^{\rm up}(b) \mid b \in \mathcal{B}(\infty)\} \subset \mathbb{C}[U^-]$ be a perfect basis, ${\bf i} \in I^r$ a reduced word for $w \in W$, and $\lambda \in P_+$.
\begin{enumerate}
\item[{\rm (1)}] The Kashiwara embedding $\Psi_{\bf i} (b)$ is equal to $-\tilde{v}_{\bf i} ^{\rm high} (\Xi^{\rm up} _w (b))$ for all $b \in \mathcal{B}_w(\infty)$.
\item[{\rm (2)}] The equalities $\widetilde{\mathcal{S}}_{\bf i} ^{(\lambda, w)} = \omega(S(X(w), \mathcal{L}_\lambda, \tilde{v}_{\bf i} ^{\rm high}, \tau_\lambda))$, $\widetilde{\mathcal{C}}_{\bf i} ^{(\lambda, w)} = \omega(C(X(w), \mathcal{L}_\lambda, \tilde{v}_{\bf i} ^{\rm high}, \tau_\lambda))$, and $\widetilde{\Delta}_{\bf i} ^{(\lambda, w)} = -\Delta(X(w), \mathcal{L}_\lambda, \tilde{v}_{\bf i} ^{\rm high}, \tau_\lambda)$ hold.
\end{enumerate}
\end{prop}\vspace{2mm}

\begin{rem}\normalfont 
In Propositions \ref{string polytopes} and \ref{polyhedral realizations}, we need not assume the property (D) for a perfect basis (see the proof of \cite[Theorem 4.1]{FN}).
\end{rem}

\section{Perfect bases with positivity properties}

In order to relate the lowest term valuations $v_{\bf i} ^{\rm low}, \tilde{v}_{\bf i} ^{\rm low}$ with the highest term valuations $v_{\bf i} ^{\rm high}, \tilde{v}_{\bf i} ^{\rm high}$, we use a perfect basis ${\bf B}^{\rm up} = \{\Xi ^{\rm up}(b) \mid b \in \mathcal{B}(\infty)\} \subset \mathbb{C}[U^-]$ that has the following positivity properties: 
\begin{enumerate}
\item[{\rm (i)}] the element $(- f_i) \cdot \Xi^{\rm up} (b)$ belongs to $\sum_{b^\prime \in \mathcal{B}(\infty)} \mathbb{R}_{\ge 0} \Xi^{\rm up} (b^\prime)$ for all $b \in \mathcal{B}(\infty)$ and $i \in I$;
\item[{\rm (ii)}] the product $\Xi^{\rm up}(\tilde{f}_i b_\infty) \cdot \Xi^{\rm up}(b)$ belongs to $\sum_{b' \in \mathcal{B}(\infty)} \mathbb{R}_{\ge 0} \Xi^{\rm up}(b')$ for all $b \in \mathcal{B}(\infty)$ and $i \in I$.
\end{enumerate}

\vspace{2mm}\begin{prop}
Let ${\bf B}^{\rm up} = \{\Xi ^{\rm up}(b) \mid b \in \mathcal{B}(\infty)\} \subset \mathbb{C}[U^-]$ be a perfect basis. The positivity properties {\rm (i)}, {\rm (ii)} are equivalent to the following positivity properties {\rm (i)$'$}, {\rm (ii)$'$}, respectively:
\begin{enumerate}
\item[{\rm (i)$'$}] the elements $(-1)^k f_i ^k \cdot \Xi^{\rm up} (b)$ and $(-1)^k \Xi^{\rm up} (b) \cdot f_i ^k$ belong to $\sum_{b^\prime \in \mathcal{B}(\infty)} \mathbb{R}_{\ge 0} \Xi^{\rm up} (b^\prime)$ for all $b \in \mathcal{B}(\infty)$, $i \in I$, and $k \ge 0$;
\item[{\rm (ii)$'$}] the product $\Xi^{\rm up}(\tilde{f}_i ^k b_\infty) \cdot \Xi^{\rm up}(b)$ belongs to $\sum_{b' \in \mathcal{B}(\infty)} \mathbb{R}_{\ge 0} \Xi^{\rm up}(b')$ for all $b \in \mathcal{B}(\infty)$, $i \in I$, and $k \ge 0$.
\end{enumerate}
\end{prop}

\begin{proof}
It follows immediately that the property (i) is equivalent to (i)$'$; hence it suffices to prove that the property (ii) implies (ii)$'$. Since $U(\mathfrak{u}^-)^\ast _{{\rm gr}, k{\bf e}_i} = \mathbb{C} \Xi^{\rm up}(\tilde{f}_i ^k b_\infty)$ for $i \in I$ and $k \ge 0$ (see Remark \ref{r:length one case}), we have $\Xi^{\rm up}(\tilde{f}_i b_\infty)^k \in \mathbb{C}^\times \Xi^{\rm up}(\tilde{f}_i ^k b_\infty)$. Then the positivity condition (ii) implies that $\Xi^{\rm up}(\tilde{f}_i b_\infty)^k \in \mathbb{R}_{>0} \Xi^{\rm up}(\tilde{f}_i ^k b_\infty)$; hence we deduce the positivity property (ii)$'$ by (ii).
\end{proof}

In the case that $\mathfrak{g}$ is of simply-laced type, Lusztig proved that  the upper global basis has the positivity properties {\rm (i)} and {\rm (ii)}, by the geometric construction of the lower global basis \cite[Theorem 11.5]{Lus_quivers}.  A desired example for general $\mathfrak{g}$ is given by the specialization of the KLR-basis at $q=1$ (see Example \ref{e:perfect_KLR}), that is, the following holds (although this proposition is an immediate consequence of \cite{KL1, KL2}, we explain a proof for the convenience of the reader).  

\vspace{2mm}\begin{prop}[{\cite{KL1, KL2}}]\label{existence with positivity}\normalfont
The specialization of the KLR-basis at $q=1$ satisfies the positivity properties {\rm (i)} and {\rm (ii)}. 
\end{prop}

\begin{proof}
As mentioned in Example \ref{e:perfect_KLR}, the KLR-basis $\{[S(b)] \mid b\in \mathcal{B}(\infty)\}$ comes from the set $\{S(b) \mid b\in \mathcal{B}(\infty)\}$ consisting of self-dual graded simple modules. 
The map $(-f_i) \cdot \colon U(\mathfrak{u}^-)^\ast _{{\rm gr}, {\bf d}}\to U(\mathfrak{u}^-)^\ast _{{\rm gr}, {\bf d}-{\bf e}_i}$ (resp., $[S(\tilde{f}_i b_\infty)]_{q=1}\cdot \colon U(\mathfrak{u}^-)^\ast _{{\rm gr}, {\bf d}}\to U(\mathfrak{u}^-)^\ast _{{\rm gr}, {\bf d}+{\bf e}_i}$) is the specialization at $q=1$ of the map induced from a certain restriction functor $\Res\colon R_{\bf d}\text{-}{\rm gmod}\to R_{{\bf d}-{\bf e}_i}\text{-}{\rm gmod}$ (resp., a certain induction functor $\Ind\colon R_{\bf d}\text{-}{\rm gmod}\to R_{{\bf d}+{\bf e}_i}\text{-}{\rm gmod}$) \cite[\S\S 2.6 and \S\S 3.1]{KL1} (see also \cite[\S\S 5.1]{KOP2}). Then, in the Grothendieck group $G_0(R_{{\bf d}\mp {\bf e}_i}\text{-}{\rm gmod})$, we have 
\begin{align*}
[\Res (S(b))]&= \sum_{b'\in \mathcal{B}(\infty), m\in\mathbb{Z}}c_{i, b}^{(m), b'}[S(b')[m]],\;\text{and}\\
[\Ind (S(b))]&= \sum_{b'\in \mathcal{B}(\infty), m\in\mathbb{Z}}d_{i, b}^{(m), b'}[S(b')[m]]
\end{align*}
for $b\in\mathcal{B}(\infty)$ and $i \in I$. Here $S(b')[m]$ denotes the grade shift of $S(b')$ by $m$, and $c_{i, b}^{(m), b'}$, $d_{i, b}^{(m), b'}$ are  the multiplicities of the corresponding simple modules in the composition series of $\Res(S(b))$ and $\Ind(S(b))$, respectively. In particular, $c_{i, b}^{(m), b'}$ and $d_{i, b}^{(m), b'}$ are nonnegative integers. Since the specialization at $q=1$ corresponds to the neglect of grade shifts, we have 
\begin{align*}
(-f_i)\cdot[S(b)]_{q=1}&=\sum\nolimits_{b'\in \mathcal{B}(\infty)}(\sum\nolimits_{m\in\mathbb{Z}}c_{i, b}^{(m), b'})[S(b')]_{q=1},\;\text{and}\\
[S(\tilde{f}_i b_\infty)]_{q=1}\cdot [S(b)]_{q=1}&=\sum\nolimits_{b'\in \mathcal{B}(\infty)}(\sum\nolimits_{m\in\mathbb{Z}}d_{i, b}^{(m), b'})[S(b')]_{q=1}
\end{align*}
in $U(\mathfrak{u}^-)^\ast _{\rm gr}\ (= \mathbb{C}[U^-])$. Hence the structure constants $\sum_{m\in\mathbb{Z}}c_{i, b}^{(m), b'}$ and $\sum_{m\in\mathbb{Z}}d_{i, b}^{(m), b'}$ are nonnegative.
\end{proof}

In the following, we will show the property (D) in Section 3 for a perfect basis ${\bf B}^{\rm up}$ with the positivity property (i). By the definition of the $U(\mathfrak{u}^-)$-bimodule structure on $U(\mathfrak{u}^-)^\ast _{\rm gr}$, we have 
\begin{align*}
(-1)^k \langle f_i ^k \cdot \Xi^{\rm up} (b), \Xi^{\rm low} (b^\prime)\rangle &= \langle \Xi^{\rm up} (b), f_i ^k \cdot \Xi^{\rm low} (b^\prime)\rangle,\ {\rm and}\\
(-1)^k \langle \Xi^{\rm up} (b) \cdot f_i ^k, \Xi^{\rm low} (b^\prime)\rangle &= \langle \Xi^{\rm up} (b), \Xi^{\rm low} (b^\prime) \cdot f_i ^k\rangle
\end{align*}
for all $b, b^\prime \in \mathcal{B}(\infty)$; hence we obtain the following. 

\vspace{2mm}\begin{lem}\label{positivity}
Let ${\bf B}^{\rm up} = \{\Xi^{\rm up}(b) \mid b \in \mathcal{B}(\infty)\} \subset \mathbb{C}[U^-]$ be a perfect basis with the positivity property $(i)$, and ${\bf B}^{\rm low} = \{\Xi^{\rm low} (b) \mid b \in \mathcal{B}(\infty)\} \subset U(\mathfrak{u}^-)$ its dual basis. Then the elements $f_i ^k \cdot \Xi^{\rm low} (b)$ and $\Xi^{\rm low} (b) \cdot f_i ^k$ belong to $\sum_{b^\prime \in \mathcal{B}(\infty)} \mathbb{R}_{\ge 0} \Xi^{\rm low} (b^\prime)$ for all $b \in \mathcal{B}(\infty)$, $i \in I$, and $k \ge 0$.
\end{lem}

\vspace{2mm}\begin{prop}\label{p:pos_D}
A perfect basis ${\bf B}^{\rm up}$ with the positivity property (i) satisfies the property (D) in Section 3.
\end{prop}\vspace{2mm}

\begin{rem}\normalfont
For the property (D), we do not need the positivity property (ii). 
\end{rem}\vspace{2mm}

\begin{proof}[Proof of Proposition \ref{p:pos_D}]
Our proof is similar to the one in \cite[\S\S 3.1 and \S\S 3.2]{Kas4}. We have $V_{w}(\lambda)=U(\mathfrak{g}_i)V_{s_{i}w}(\lambda)$ if $\ell(s_{i}w)<\ell (w)$ (see \cite[Lemma 3.2.1]{Kas4}). Hence it suffices to prove the following claim:

\begin{clm} If a $U(\mathfrak{u}_i^+)$-submodule $N$ $(\mathfrak{u}_i^+ \coloneqq \mathbb{C}e_i)$ of $V(\lambda)$ is spanned by $\{\Xi_{\lambda}^{\rm low}(b) \mid b\in \mathcal{B}_N\}$ for some subset $\mathcal{B}_N\subset \mathcal{B}(\lambda)$, then $U(\mathfrak{g}_i)N=\sum_{b\in \tilde{\mathcal{B}}_N^{(i)}}\mathbb{C}\Xi_{\lambda}^{\rm low}(b)$, where $\tilde{\mathcal{B}}_N^{(i)} \coloneqq \bigcup_{k\in \mathbb{Z}_{\ge 0}}\tilde{f}_i ^k\mathcal{B}_N\setminus \{0\}$.
\end{clm}

\noindent For a $\mathbb{C}$-subspace $M\subset V(\lambda)$ (resp., a subset $\mathcal{S}\subset\mathcal{B}(\lambda)$), write $W_i^{\ell}(M) \coloneqq W_i^{\ell}(V(\lambda))\cap M$ and $I_i^{\ell}(M) \coloneqq I_i^{\ell}(V(\lambda))\cap M$ (resp., $W_i^{\ell}(\mathcal{S}) \coloneqq W_i^{\ell}(\mathcal{B}(\lambda))\cap \mathcal{S}$) (see Proposition \ref{p:compatibility}). By Proposition \ref{p:compatibility} {\rm (2)}, $W_i^{\ell}(N)=\sum_{b\in W_i^{\ell}(\mathcal{B}_N)}\mathbb{C}\Xi_{\lambda}^{\rm low}(b)$, and 
\[
\mathcal{B}_N=\mathcal{B}_N\cap\left(\bigcup\nolimits_{k\in \mathbb{Z}_{\ge 0}}\tilde{f}_i ^k\{b\in \mathcal{B}_N\mid \varepsilon_i(b)=0\}\right).
\]
Let $\ell_0$ be the maximum integer $\ell$ such that $W_i^{\ell}(U(\mathfrak{g}_i)N)\neq 0$. Then 
\begin{align*}
W_i^{\ell_0}(U(\mathfrak{g}_i)N)=I_i^{\ell_0}(U(\mathfrak{g}_i)N),\;\;W_i^{\ell_0}(U(\mathfrak{g}_i)N)\cap \Ker e_i=W_i^{\ell_0}(N)\cap \Ker e_i.
\end{align*}
Therefore 
\[
W_i^{\ell_0}(U(\mathfrak{g}_i)N)\cap \Ker e_i=\sum\nolimits_{b\in \mathcal{B}_N;\ \langle\wt (b), h_i\rangle=\ell_0}\mathbb{C}\Xi_{\lambda}^{\rm low}(b).
\]
Moreover $\varepsilon_i(b)=0$ for all $b\in \mathcal{B}_N$ with $\langle\wt (b), h_i\rangle=\ell_0$. By the way, for $v=\sum_{b\in \mathcal{B}(\lambda)}c_b\Xi_{\lambda}^{\rm low}(b)\in V(\lambda)$ with $c_b\in \mathbb{R}_{\ge 0}$, 
\[
f_i^k \cdot v \neq 0\;\text{for}\;k\le \max\{\varphi_i(b)\mid c_b\neq 0 \}
\]
by \eqref{f_action} and Lemma \ref{positivity}. Hence we can deduce that $f_i^k \cdot \Xi_{\lambda}^{\rm low}(b)\in \mathbb{C}^{\times}\Xi_{\lambda}^{\rm low}(\tilde{f}_i^k b)$ for all $b\in \mathcal{B}_N$ with $\langle\wt(b), h_i\rangle=\ell_0$ by \eqref{f_action}. Therefore $W_i^{\ell_0}(U(\mathfrak{g}_i)N)=\sum_{b\in W_i^{\ell_0}(\tilde{\mathcal{B}}_N^{(i)})}\mathbb{C}\Xi_{\lambda}^{\rm low}(b)$. By using descending induction on $\ell$ and replacing $U(\mathfrak{g}_i)N$ by $U(\mathfrak{g}_i)N/W_i^{\ell+1}(U(\mathfrak{g}_i)N)=U(\mathfrak{g}_i)N/U(\mathfrak{g}_i)W_i^{\ell+1}(N)$ (see \cite[Lemma 3.1.4]{Kas4}) in the argument above, we obtain the assertion.
\end{proof}

For $w \in W$, take $i \in I$ (resp., $i^\prime \in I$) such that $\ell(s_i w) < \ell(w)$ (resp., $\ell(w s_{i^\prime}) < \ell(w)$). Then the left action of ${\mathfrak u}_i ^-$ (resp., the right action of ${\mathfrak u}_{i^\prime} ^-$) on $\mathbb{C}[U^-]$ induces a left action of ${\mathfrak u}_i ^-$ (resp., a right action of ${\mathfrak u}_{i^\prime} ^-$) on $\mathbb{C}[U^- \cap X(w)]$ by the restriction map $\eta_w: \mathbb{C}[U^-] \twoheadrightarrow \mathbb{C}[U^- \cap X(w)]$. The following is an immediate consequence of Corollary \ref{vanishing} (4) (see also Proposition \ref{p:pos_D}) and the positivity properties (i)$'$, (ii)$'$.

\vspace{2mm}\begin{cor}\label{positivity corollary}
For $w \in W$, the following hold.
\begin{enumerate}
\item[{\rm (1)}] The elements $(-1)^k f_i ^k \cdot \Xi_w ^{\rm up} (b)$ and $(-1)^k \Xi_w ^{\rm up} (b) \cdot f_{i^ \prime} ^k$ belong to $\sum_{b^\prime \in \mathcal{B}_w(\infty)} \mathbb{R}_{\ge 0} \Xi_w ^{\rm up} (b^\prime)$ for all $b \in \mathcal{B}_w(\infty)$, $i, i^{\prime} \in I$, and $k \ge 0$ such that $\ell(s_i w) < \ell(w)$, $\ell(w s_{i^\prime}) < \ell(w)$.
\item[{\rm (2)}] The product $\Xi_w ^{\rm up}(\tilde{f}_i ^k b_\infty) \cdot \Xi_w ^{\rm up}(b)$ belongs to $\sum_{b' \in \mathcal{B}_w(\infty)} \mathbb{R}_{\ge 0} \Xi_w ^{\rm up}(b')$ for all $b \in \mathcal{B}_w(\infty)$, $i \in I$, and $k \ge 0$.
\end{enumerate}
\end{cor}\vspace{2mm}

Let ${\bf i} = (i_1, \ldots, i_r)$ be a reduced word for $w \in W$, and regard the coordinate ring $\mathbb{C}[U^- \cap X(w)]$ as a $\mathbb{C}$-subalgebra of $\mathbb{C}[U_{i_1} ^- \times \cdots \times U_{i_r} ^-] = \mathbb{C}[t_1, \ldots, t_r]$. 

\vspace{2mm}\begin{prop}\label{positivity proposition}
The coefficient of $t_1 ^{a_1} \cdots t_r ^{a_r}$ in $\Xi^{\rm up} _w(b)$ is a nonnegative real number for every $b \in \mathcal{B}_w(\infty)$ and $a_1, \ldots, a_r \in \mathbb{Z}_{\ge 0}$.
\end{prop}

\begin{proof}
For $b \in \mathcal{B}_w(\infty)$ and $a_1, \ldots, a_r \in \mathbb{Z}_{\ge 0}$, denote by $A_b ^{(a_1, \ldots, a_r)} \in \mathbb{C}$ the coefficient of $t_1 ^{a_1} \cdots t_r ^{a_r}$ in $\Xi^{\rm up} _w(b)$. Then we know from the equality (\ref{derivation}) in Section 2 that $A_b ^{(a_1, \ldots, a_r)}$ is equal to the value \[\frac{(-1)^{a_1 + \cdots + a_r}}{a_1! \cdots a_r!}  (f_{i_r} ^{a_r} \cdot (\cdots (f_{i_2} ^{a_2} \cdot (f_{i_1} ^{a_1} \cdot \Xi^{\rm up} _w(b))|_{t_1 = 0})|_{t_2 = 0} \cdots)|_{t_{r-1} = 0})|_{t_r = 0}.\] If we write $w_{\ge k} \coloneqq s_{i_k} s_{i_{k+1}} \cdots s_{i_r}$ for $1 \le k \le r$ and $w_{\ge r+1} \coloneqq e$, the identity element of $W$, then the restriction map $\eta_{k, k+1} \colon \mathbb{C}[U^- \cap X(w_{\ge k})] \twoheadrightarrow \mathbb{C}[U^- \cap X(w_{\ge k+1})]$ is given by $t_k \mapsto 0$; hence we obtain the equality \[A_b ^{(a_1, \ldots, a_r)} = \frac{(-1)^{a_1 + \cdots + a_r}}{a_1! \cdots a_r!} \eta_{r, r+1} (f_{i_r} ^{a_r} \cdot (\eta_{r-1, r}(\cdots(\eta_{2, 3}(f_{i_2} ^{a_2} \cdot(\eta_{1, 2}(f_{i_1} ^{a_1} \cdot \Xi^{\rm up} _w(b)))))\cdots))),\] where we identify the coordinate ring $\mathbb{C}[U^- \cap X(w_{\ge r+1})] = \mathbb{C}[U^- \cap X(e)]$ with $\mathbb{C}$ by $\Xi^{\rm up} _e(b_\infty) \mapsto 1$ (note the equality $\mathcal{B}_e(\infty) = \{b_\infty\}$). Now by using Corollaries \ref{vanishing} (4) and \ref{positivity corollary} (1) repeatedly, we conclude that $A_b ^{(a_1, \ldots, a_r)}$ is a nonnegative real number. This proves the proposition.
\end{proof}

\section{Main result}

We write ${\bf a}^{\rm op} \coloneqq (a_r, \ldots, a_1)$ for an element ${\bf a} = (a_1, \ldots, a_r) \in \mathbb{R}^r$, and $H^{\rm op} \coloneqq \{{\bf a}^{\rm op} \mid {\bf a} \in H\}$ for a subset $H \subset \mathbb{R}^r$. The following is the main result of this paper.

\vspace{2mm}\begin{thm}\label{main result}
Let ${\bf i} \in I^r$ be a reduced word for $w \in W$, and ${\bf B}^{\rm up} = \{\Xi^{\rm up} (b) \mid b \in \mathcal{B}(\infty)\} \subset \mathbb{C}[U^-]$ a perfect basis with the positivity properties (i), (ii) in Section 4. Then the following equalities hold:
\begin{align*}
v_{\bf i} ^{\rm low} (\Xi^{\rm up} _w(b)) &= - \tilde{v}_{\bf i} ^{\rm high} (\Xi^{\rm up} _w(b))^{\rm op}\ {\it and}\\
\tilde{v}_{\bf i} ^{\rm low} (\Xi^{\rm up} _w(b)) &= - v_{\bf i} ^{\rm high} (\Xi^{\rm up} _w(b))^{\rm op}
\end{align*} 
for all $b \in \mathcal{B}_w(\infty)$.
\end{thm}\vspace{2mm}

Before proving Theorem \ref{main result}, we give some corollaries. The following corollary (Corollary \ref{first corollary}) is an immediate consequence of Corollary \ref{vanishing} (2) and Theorem \ref{main result}.

\vspace{2mm}\begin{cor}\label{first corollary}
Let ${\bf i} \in I^r$ be a reduced word for $w \in W$, $\lambda \in P_+$, and ${\bf B}^{\rm up} = \{\Xi^{\rm up} (b) \mid b \in \mathcal{B}(\infty)\} \subset \mathbb{C}[U^-]$ a perfect basis with the positivity properties (i), (ii) in Section 4. Then the following equalities hold:
\begin{align*}
v_{\bf i} ^{\rm low} (\Xi_{\lambda, w} ^{\rm up} (b)/\tau_\lambda) &= - \tilde{v}_{\bf i} ^{\rm high} (\Xi_{\lambda, w} ^{\rm up} (b)/\tau_\lambda)^{\rm op}\ {\it and}\\
\tilde{v}_{\bf i} ^{\rm low} (\Xi_{\lambda, w} ^{\rm up} (b)/\tau_\lambda) &= - v_{\bf i} ^{\rm high} (\Xi_{\lambda, w} ^{\rm up} (b)/\tau_\lambda)^{\rm op}
\end{align*} 
for all $b \in \mathcal{B}_w(\lambda)$.
\end{cor}\vspace{2mm}

\vspace{2mm}\begin{cor}\label{corollary of main result}
Let ${\bf i} \in I^r$ be a reduced word for $w \in W$, and $\lambda \in P_+$. Then the following equalities hold:
\begin{align*}
&\Delta(X(w), \mathcal{L}_\lambda, v_{\bf i} ^{\rm low}, \tau_\lambda) = -\Delta(X(w), \mathcal{L}_\lambda, \tilde{v}_{\bf i} ^{\rm high}, \tau_\lambda)^{\rm op}\ {\it and}\\
&\Delta(X(w), \mathcal{L}_\lambda, \tilde{v}_{\bf i} ^{\rm low}, \tau_\lambda) = -\Delta(X(w), \mathcal{L}_\lambda, v_{\bf i} ^{\rm high}, \tau_\lambda)^{\rm op}.
\end{align*}
\end{cor}

\begin{proof}
We prove only the assertion for $v_{\bf i} ^{\rm low}$ and $\tilde{v}_{\bf i} ^{\rm high}$; a proof of the assertion for $\tilde{v}_{\bf i} ^{\rm low}$ and $v_{\bf i} ^{\rm high}$ is similar. Let $\{\Xi^{\rm up} (b) \mid b \in \mathcal{B}(\infty)\}$ be a perfect basis with the positivity properties (i), (ii); the existence of such a perfect basis is guaranteed by Proposition \ref{existence with positivity}. We see by Corollary \ref{first corollary} and Proposition \ref{polyhedral realizations} (1) that \[v_{\bf i} ^{\rm low} (\Xi_{\lambda, w} ^{\rm up} (\pi_\lambda(b))/\tau_\lambda) = -\tilde{v}_{\bf i} ^{\rm high} (\Xi_{\lambda, w} ^{\rm up} (\pi_\lambda(b))/\tau_\lambda)^{\rm op} = \Psi_{\bf i}(b)^{\rm op}\] for all $b \in \widetilde{\mathcal{B}}_w(\lambda)$. Remark that $\{\Xi_{\lambda, w} ^{\rm up} (b) \mid b \in \mathcal{B}_w(\lambda)\}$ is a $\mathbb{C}$-basis of $H^0(X(w), \mathcal{L}_\lambda)$, and that $-\tilde{v}_{\bf i} ^{\rm high} (\Xi_{\lambda, w} ^{\rm up} (\pi_\lambda(b))/\tau_\lambda) = \Psi_{\bf i}(b)$, $b \in \widetilde{\mathcal{B}}_w(\lambda)$, are all distinct. Hence we see by Proposition \ref{prop1,val} that \[\{v_{\bf i} ^{\rm low} (\sigma/\tau_\lambda) \mid \sigma \in H^0(X(w), \mathcal{L}_\lambda) \setminus \{0\}\} = \{-\tilde{v}_{\bf i} ^{\rm high} (\sigma/\tau_\lambda)^{\rm op} \mid \sigma \in H^0(X(w), \mathcal{L}_\lambda) \setminus \{0\}\}.\] This implies the assertion by the definition of Newton-Okounkov bodies (see also the proof of \cite[Corollary 4.2]{FN}).
\end{proof}
Since $v_{\bf i} ^{\rm low} = v_{X(w_{\ge \bullet})}$ and $\tilde{v}_{\bf i} ^{\rm low} = v_{X(w_{\le \bullet})}$, we see that Theorem and Corollary in Introduction follow immediately by Propositions \ref{string polytopes}, \ref{polyhedral realizations} and Corollaries \ref{vanishing} (2), \ref{first corollary}, \ref{corollary of main result} (see also Remark \ref{parametrization for modules}). 

The following corollaries (Corollaries \ref{corollary rationality for lowest} and \ref{corollary integral points for lowest}) are immediate consequences of Propositions \ref{string lattice points} (2), \ref{polyhedral lattice points} (2), \ref{string polytopes}, \ref{polyhedral realizations}, Theorem \ref{main result}, and Corollary \ref{corollary of main result}.

\vspace{2mm}\begin{cor}\label{corollary rationality for lowest}
The Newton-Okounkov bodies $\Delta(X(w), \mathcal{L}_\lambda, v_{\bf i} ^{\rm low}, \tau_\lambda)$ and $\Delta(X(w), \mathcal{L}_\lambda, \tilde{v}_{\bf i} ^{\rm low}, \tau_\lambda)$ are both rational convex polytopes.
\end{cor}

\vspace{2mm}\begin{cor}\label{corollary integral points for lowest}
Let $\{\Xi^{\rm up} (b) \mid b \in \mathcal{B}(\infty)\} \subset \mathbb{C}[U^-]$ be  a perfect basis with the positivity properties (i), (ii). Then the following equalities hold:
\begin{align*}
&\Delta(X(w), \mathcal{L}_\lambda, v_{\bf i} ^{\rm low}, \tau_\lambda) \cap \mathbb{Z}^r = \{v_{\bf i} ^{\rm low} (\Xi^{\rm up} _w(b)) \mid b \in \widetilde{\mathcal{B}}_w (\lambda)\}\ {\it and}\\
&\Delta(X(w), \mathcal{L}_\lambda, \tilde{v}_{\bf i} ^{\rm low}, \tau_\lambda) \cap \mathbb{Z}^r = \{\tilde{v}_{\bf i} ^{\rm low} (\Xi^{\rm up} _w(b)) \mid b \in \widetilde{\mathcal{B}}_w (\lambda)\}.
\end{align*}
\end{cor}\vspace{2mm}

In a way similar to our proof of Corollary \ref{corollary of main result}, we see that the following equalities hold:
\begin{align*}
&S(X(w), \mathcal{L}_\lambda, v_{\bf i} ^{\rm low}, \tau_\lambda) = \{(k, -{\bf a}^{\rm op}) \mid (k, {\bf a}) \in S(X(w), \mathcal{L}_\lambda, \tilde{v}_{\bf i} ^{\rm high}, \tau_\lambda)\}\ {\rm and}\\
&S(X(w), \mathcal{L}_\lambda, \tilde{v}_{\bf i} ^{\rm low}, \tau_\lambda) = \{(k, -{\bf a}^{\rm op}) \mid (k, {\bf a}) \in S(X(w), \mathcal{L}_\lambda, v_{\bf i} ^{\rm high}, \tau_\lambda)\}.
\end{align*}
Hence we obtain the following by Propositions \ref{string lattice points} (1), \ref{polyhedral lattice points} (1), \ref{string polytopes} (2), \ref{polyhedral realizations} (2) and Gordan's lemma (see, for instance, \cite[Proposition 1.2.17]{CLS}).

\vspace{2mm}\begin{cor}
The semigroups $S(X(w), \mathcal{L}_\lambda, v_{\bf i} ^{\rm low}, \tau_\lambda)$ and $S(X(w), \mathcal{L}_\lambda, \tilde{v}_{\bf i} ^{\rm low}, \tau_\lambda)$ are both finitely generated.
\end{cor}\vspace{2mm}

\begin{proof}[Proof of Theorem \ref{main result}]
We prove only the assertion for the valuations $\tilde{v}^{\rm low} _{\bf i}$ and $v_{\bf i} ^{\rm high}$; a proof of the other claim is similar. Write ${\bf i} = (i_1, \ldots, i_r)$ and $\Phi_{\bf i} (b) = (a_1, \ldots, a_r)$ for $b \in \mathcal{B}_w (\infty)$. We first consider the case $b \in \mathcal{B}_{s_{i_1}}(\infty)$. In this case, there exists $a \in \mathbb{Z}_{\ge 0}$ such that $b = \tilde{f}_{i_1} ^a b_\infty$. Then we deduce from the definition of $\Phi_{\bf i}$ that
\begin{align*}
-v^{\rm high} _{\bf i}(\Xi^{\rm up} _w(b)) &= \Phi_{\bf i} (b)\quad({\rm by\ Proposition}\ \ref{string polytopes}\ (1))\\
&= (a, 0, \ldots, 0).
\end{align*}
Hence it follows from the definition of $v^{\rm high} _{\bf i}$ that $\Xi^{\rm up} _w(b) = c t_1 ^a + ({\rm other\ terms})$ for some $c \neq 0$, where ``other terms'' means a linear combination of monomials that are not equal to $t_1 ^a$. Since $\Xi^{\rm up} (\tilde{f}_{i_1} ^a b_\infty) \in U(\mathfrak{u}^-)^\ast _{{\rm gr}, a{\bf e}_{i_1}}$, it follows that all monomials in ``other terms'' are of degree $a$, and hence that they contain $t_k$ for some $2 \le k \le r$ as variables. By the definition of $\tilde{v}^{\rm low} _{\bf i}$, this implies that 
\begin{align*}
\tilde{v}^{\rm low} _{\bf i}(\Xi^{\rm up} _w(b)) &= (0, \ldots, 0, a)\\
&= (a, 0, \ldots, 0)^{\rm op}\\
&= -v^{\rm high} _{\bf i}(\Xi^{\rm up} _w(b))^{\rm op}.
\end{align*}
Next we proceed by induction on $r = \ell(w)$. The case $r = 1$ is already included in the special case above. Assume that $r\ge 2$. Let us consider the special case $a_1=0$. Then $b$ is an element of $\mathcal{B}_{w_{\ge 2}} (\infty)$, where $w_{\ge 2} \coloneqq s_{i_2} \cdots s_{i_r}$; furthermore, by the definition of $v_{\bf i} ^{\rm high}$, the equality $a_1 = 0$ implies that $t_1$ does not appear in $\Xi^{\rm up} _w(b) \in \mathbb{C}[t_1, \ldots, t_r]$. Hence we deduce that
\begin{align*}
\tilde{v}_{\bf i} ^{\rm low} (\Xi^{\rm up} _w(b)) &= (\tilde{v}_{{\bf i}_{\ge 2}} ^{\rm low} (\Xi^{\rm up} _{w_{\ge 2}} (b)), 0)\\
&= -(v_{{\bf i}_{\ge 2}} ^{\rm high}(\Xi^{\rm up} _{w_{\ge 2}} (b))^{\rm op}, 0)\quad({\rm by\ induction\ hypothesis})\\
&= -(0, v_{{\bf i}_{\ge 2}} ^{\rm high}(\Xi^{\rm up} _{w_{\ge 2}} (b)))^{\rm op}\\
&= -v^{\rm high} _{\bf i}(\Xi^{\rm up} _w(b))^{\rm op},
\end{align*}
where ${\bf i}_{\ge 2} \coloneqq (i_2, \ldots, i_r)$, a reduced word for $w_{\ge 2}$. 

Now we consider a total order $\lhd$ on $\mathcal{B}_w (\infty)$ defined as follows: 
\[
b'\lhd b''\;\text{if and only if}\;
\begin{cases}
|b'|<|b''|,\;\text{or}\\
|b'|=|b''|\;\text{and}\;\Phi_{\bf i} (b')<\Phi_{\bf i} (b'')\;\text{with respect to the lexicographic order}\;<;
\end{cases}
\]
here $|b'''|:=\sum_{i\in I}d_i$ for $b'''\in\mathcal{B}_w (\infty)$ with $\wt (b''')=-\sum_{i\in I}d_i\alpha_i$. We prove the assertion by induction on $b$ with respect to the total order $\lhd$. We only have to consider the case that $b$ does not belong to the two special cases above. Namely, assume that $b \notin \mathcal{B}_{s_{i_1}} (\infty)$ and $a_1 > 0$. Set $b_1 \coloneqq \tilde{f}_{i_1} ^{a_1} b_\infty$ and $b_2 \coloneqq \tilde{f}_{i_2} ^{a_2} \cdots \tilde{f}_{i_r} ^{a_r} b_\infty$. Then we have $\Phi_{\bf i} (b_1) = (a_1, 0, \ldots, 0)$ and $\Phi_{\bf i} (b_2) = (0, a_2, \ldots, a_r)$.
Hence it follows that 
\begin{equation}
\begin{aligned}\label{highest equation 1}
v^{\rm high} _{\bf i}(\Xi^{\rm up} _w(b)) &= -(a_1, \ldots, a_r)\\
&= -\Phi_{\bf i} (b_1) -\Phi_{\bf i} (b_2)\\ 
&= v^{\rm high} _{\bf i}(\Xi^{\rm up} _w (b_1)) + v^{\rm high} _{\bf i}(\Xi^{\rm up} _w (b_2))\quad({\rm by\ Proposition}\ \ref{string polytopes}\ (1)).
\end{aligned}
\end{equation}
Now we deduce from the results for the two special cases that
\begin{equation}
\begin{aligned}\label{highest equation 2}
v^{\rm high} _{\bf i}(\Xi^{\rm up} _w (b_1)) + v^{\rm high} _{\bf i}(\Xi^{\rm up} _w (b_2)) &= -(\tilde{v}_{\bf i} ^{\rm low}(\Xi^{\rm up} _w (b_1))^{\rm op} + \tilde{v}_{\bf i} ^{\rm low}(\Xi^{\rm up} _w (b_2))^{\rm op})\\
&= - \tilde{v}_{\bf i} ^{\rm low}(\Xi^{\rm up} _w (b_1) \cdot \Xi^{\rm up} _w (b_2))^{\rm op}\quad({\rm by\ Definition}\ \ref{def,val}).
\end{aligned}
\end{equation}
From these, it suffices to prove the equality $\tilde{v}_{\bf i} ^{\rm low}(\Xi^{\rm up} _w (b_1) \cdot \Xi^{\rm up} _w (b_2)) = \tilde{v}_{\bf i} ^{\rm low}(\Xi^{\rm up} _w(b))$. Now we know from Corollary \ref{positivity corollary} (2) that
\begin{equation}
\begin{aligned}\label{expansion}
\Xi^{\rm up} _w (b_1) \cdot \Xi^{\rm up} _w (b_2) = \sum_{b_3 \in \mathcal{B}_w (\infty)} C_{b_1, b_2} ^{(b_3)} \cdot \Xi^{\rm up} _w (b_3) 
\end{aligned}
\end{equation}
for some $C_{b_1, b_2} ^{(b_3)} \in \mathbb{R}_{\ge 0}$, $b_3 \in \mathcal{B}_w (\infty)$. Since $C_{b_1, b_2} ^{(b_3)}$ is nonnegative for all $b_3 \in \mathcal{B}_w (\infty)$, Proposition \ref{positivity proposition} implies that any cancellation of monomials does not occur in the right hand side of (\ref{expansion}). From this, we deduce by the definition of $\tilde{v}_{\bf i} ^{\rm low}$ that \[\tilde{v}_{\bf i} ^{\rm low} (\Xi^{\rm up} _w (b_1) \cdot \Xi^{\rm up} _w (b_2)) = \min\{\tilde{v}_{\bf i} ^{\rm low}(\Xi^{\rm up} _w (b_3)) \mid b_3 \in \mathcal{B}_w (\infty),\ C_{b_1, b_2} ^{(b_3)} \neq 0\},\] where ``min'' means the minimum with respect to the lexicographic order $<$. Similarly, we see that
\begin{equation}
\begin{aligned}\label{highest valuation of product}
-v_{\bf i} ^{\rm high} (\Xi^{\rm up} _w (b_1) \cdot \Xi^{\rm up} _w (b_2)) = \max\{-v_{\bf i} ^{\rm high}(\Xi^{\rm up} _w (b_3)) \mid b_3 \in \mathcal{B}_w (\infty),\ C_{b_1, b_2} ^{(b_3)} \neq 0\},
\end{aligned}
\end{equation}
where ``max'' means the maximum with respect to the lexicographic order $<$. Recall that $-v_{\bf i} ^{\rm high}(\Xi^{\rm up} _w (b_3)) = \Phi_{\bf i} (b_3)$, $b_3 \in \mathcal{B}_w(\infty)$, are all distinct. Hence equations (\ref{highest equation 1}) and (\ref{highest valuation of product}) imply that $C_{b_1, b_2} ^{(b)} \neq 0$, and that if $C_{b_1, b_2} ^{(b_3)} \neq 0$ and $b_3 \neq b$, then $\wt (b)=\wt (b_3)$ and $-v_{\bf i} ^{\rm high}(\Xi^{\rm up} _w (b_3)) < -v_{\bf i} ^{\rm high}(\Xi^{\rm up} _w(b))$; in particular, it holds that 
\begin{align*}
\tilde{v}_{\bf i} ^{\rm low}(\Xi^{\rm up} _w (b_1) \cdot \Xi^{\rm up} _w (b_2))^{\rm op} &= -v_{\bf i} ^{\rm high}(\Xi^{\rm up} _w(b))\quad({\rm by\ equations}\ (\ref{highest equation 1})\ {\rm and}\ (\ref{highest equation 2}))\\
&> -v_{\bf i} ^{\rm high}(\Xi^{\rm up} _w (b_3))\\
&= \tilde{v}_{\bf i} ^{\rm low} (\Xi^{\rm up} _w (b_3))^{\rm op}\quad({\rm by\ induction\ hypothesis\ concerning}\ b).
\end{align*}
From these, we obtain that $\tilde{v}_{\bf i} ^{\rm low} (\Xi^{\rm up} _w (b_1) \cdot \Xi^{\rm up} _w (b_2)) = \tilde{v}_{\bf i} ^{\rm low}(\Xi^{\rm up} _w(b))$. This proves the theorem.
\end{proof}

\begin{rem}\normalfont
Since Corollary \ref{corollary of main result} follows from Corollary \ref{first corollary}, it is natural to ask why we consider not only $\{\Xi^{\rm up} _{\lambda, w} (b) \mid b \in \mathcal{B}_w(\lambda)\}$ but also $\{\Xi^{\rm up} _w(b) \mid b \in \mathcal{B}_w(\infty)\}$. The reason is that in order to prove the assertion of Corollary \ref{first corollary} for $\{\Xi^{\rm up} _{\lambda, w} (b) \mid b \in \mathcal{B}_w(\lambda)\} \subset H^0(X(w), \mathcal{L}_\lambda)$, we have to consider an element of $\mathbb{C}[U^- \cap X(w)]$ that does not belong to $\iota_\lambda(H^0(X(w), \mathcal{L}_\lambda))$. In our proof of Theorem \ref{main result}, we use the elements $b_1, b_2\in \mathcal{B}_w(\infty)$ determined from $b\in \mathcal{B}_w(\infty)$ with $b \notin \mathcal{B}_{s_{i_1}} (\infty)$ and $a_1 > 0$. An important point is that, even if $b \in \widetilde{\mathcal{B}}_w(\lambda)$ for some $\lambda \in P_+$, the element $b_1$ is not necessarily an element of $\widetilde{\mathcal{B}}_w(\lambda)$. Let us see this with a specific example. Take $G, {\bf i}, \lambda$ as in Example \ref{main example}. Then the set $\Phi_{\bf i}(\mathcal{B}(\lambda)) = \Phi_{\bf i}(\widetilde{\mathcal{B}}(\lambda))$ of string parametrizations is identical to \[\{(0, 0, 0), (1, 0, 0), (0, 1, 0), (1, 1, 0), (0, 1, 1), (2, 1, 0), (0, 2, 1), (1, 2, 1)\}.\] For $b \in \widetilde{\mathcal{B}}(\lambda)$ such that $\Phi_{\bf i}(b) = (2, 1, 0)$, the element $b_1 \in \mathcal{B}(\infty)$ satisfies $\Phi_{\bf i}(b_1) = (2, 0, 0)$, which implies that $b_1 \notin \widetilde{\mathcal{B}}(\lambda)$. 
\end{rem}

\end{document}